\documentclass[reqno,a4paper,11pt]{amsart}

\usepackage[english]{babel}

\usepackage{amsfonts}
\usepackage{amsmath}
\usepackage{amsthm}
\usepackage{amssymb}
\usepackage{indentfirst}
\usepackage{xcolor}
\usepackage{tabularx}
\usepackage{graphicx}
\usepackage{esint}
\usepackage{dsfont}

\allowdisplaybreaks

\usepackage{soul}

\usepackage[a4paper,top=3cm,bottom=3cm,left=2.5cm,right=2.5cm]{geometry}

\usepackage{subcaption}

\usepackage{mathrsfs}
\usepackage{mathtools}
\usepackage{enumitem}
\usepackage[normalem]{ulem}

\usepackage[utf8]{inputenc}
\usepackage[T1]{fontenc}
\usepackage[english]{babel}

\numberwithin{equation}{section}

\theoremstyle{plain}
\begingroup
\theoremstyle{plain}
\newtheorem{theorem}{Theorem}[section]
\newtheorem{corollary}[theorem]{Corollary}
\newtheorem{proposition}[theorem]{Proposition}
\newtheorem{lemma}[theorem]{Lemma}
\theoremstyle{definition}
\newtheorem{definition}[theorem]{Definition}

\theoremstyle{remark}
\newtheorem{remark}[theorem]{Remark}
\endgroup

\theoremstyle{definition}
\theoremstyle{remark}

\mathchardef\emptyset="001F

\newcommand{\R}{\mathbb{R}}

\newcommand{\dive}{\mathrm{div}}

\newcommand{\argmin}{\mathrm{argmin}}

\newcommand{\bb}{\boldsymbol{b}}

\newcommand{\de}{\mathrm{d}}

\newcommand{\by}{\boldsymbol{y}}

\newcommand{\cH}{\mathcal{H}}

\newcommand{\cP}{\mathcal{P}}

\newcommand{\cT}{\mathcal{T}}

\newcommand{\bell}{\boldsymbol{\ell}}

\definecolor{dred}{rgb}{.8,0,0}
\definecolor{ddmagenta}{rgb}{0.7,0,0.9}
\definecolor{ddcyan}{rgb}{0,0.2,1.0}
\definecolor{Orchid}{rgb}{0.7,0.4,0}
\definecolor{blue_links}{RGB}{13,0,180} 
\definecolor{lightblue}{RGB}{0.9,0.9,1}

\newcommand{\eps}{\varepsilon}

\usepackage{hyperref}
\hypersetup{
    colorlinks=true, 
    linktoc=all,     
    linkcolor=blue_links,  
    citecolor=blue_links,
    urlcolor=blue_links,
}

\newcommand{\cR}{\mathcal{R}}

\newcommand{\norm}[1]{\lVert #1 \rVert}

\makeatletter
\@namedef{subjclassname@2020}{%
  \textup{2020} Mathematics Subject Classification}
\makeatother

\title{Mean-field limits for entropic multi-population dynamical systems}

\author[S. Almi]{Stefano Almi}
\address[Stefano Almi]{Institute of Analysis and Scientific Computing, TU Wien, Wiedner Hauptstra\ss e 8-10, 1040 Vienna, Austria \& Dipartimento di Matematica e Applicazioni ``R.~Caccioppoli'',
Universit\`a di Napoli Federico II, via Cintia, 80126 Napoli, Italy}
\email{stefano.almi@unina.it}

\author[C. D'Eramo]{Claudio D'Eramo}
\address[Claudio D'Eramo]{4S Group, Corso Peschiera 146, 10138 Torino, Italy} 
\email{deramo.claudio@gmail.com}

\author[M. Morandotti]{Marco Morandotti}
\address[Marco Morandotti]{Dipartimento di Scienze Matematiche ``G.~L.~Lagrange'',
Politecnico di Torino, Corso Duca degli Abruzzi 24,
10129 Torino, Italy.}
\email{marco.morandotti@polito.it}

\author[F. Solombrino]{Francesco Solombrino}
\address[Francesco Solombrino]{Dipartimento di Matematica e Applicazioni ``R.~Caccioppoli'',
Universit\`a di Napoli Federico II, via Cintia, 80126 Napoli, Italy.}
\email{francesco.solombrino@unina.it}

\date{\today}


\begin{document}

\begin{abstract}
The well-posedness of a multi-population dynamical system with an entropy regularization and its convergence to a suitable mean-field approximation are proved, under a general set of assumptions. Under further assumptions on the evolution of the labels, the case of different time scales between the agents' locations and labels dynamics is considered. The limit system couples a mean-field-type evolution in the space of positions and an instantaneous optimization of the payoff functional in the space of labels.

\end{abstract}

\maketitle

\tableofcontents

%

\section{Introduction}
\noindent\textbf{Overview of the topic.}
After being introduced in statistical physics by Kac \cite{Kac} and then by McKean \cite{McKean} to describe the collisions between particles in a gas, the mean-field approximation has become a powerful tool to analyze the asymptotic behavior of systems of interacting agents in biology, sociology, and economics. We may mention, \emph{e.g.},  recent applications to the description of cell aggregation and motility \cite{motility2,motility1}, coordinated animal motion \cite{animal},  cooperative robots \cite{robots}, and influence of key investors in the stock market \cite[Introduction]{stock}. 

The modeling of these systems is usually inspired  from Newtonian laws of motion and is based on pairwise forces accounting for repulsion/attraction, alignment, self-propulsion/friction in biological, social, or economical interactions. In this way, the evolution of~$N$ agents with time-dependent locations, $x^1_t,\dots,x^N_t$ in $\mathbb{R}^d$ is described by the ODE system 
\begin{equation*}
    \dot{x}^i_t = \frac{1}{N}\sum\limits_{j=1}^N f(x^i_t,x^j_t) \quad\textrm{for }i=1,\dots,N,\,\, t\in (0,T],
\end{equation*}
where $f$ is a pre-determined pairwise interaction force between pairs of agents. The above first-order structure of multi-agent interactions appears, for instance, in some recent model in opinion formation \cite{opinion}, vehicular traffic flow \cite{traffic}, pedestrian motion \cite{Tosin}, and synchronisation of chemical and biological oscillators in neuroscience \cite{chem}.

Another context, where this approach has proved to be a useful one, is that of evolutionary games, where players are simultaneously willing to optimize their cost: this includes game theoretic models of evolution \cite{Josef} or mean-field games (\cite{Malhame,Jean}) in order to describe \emph{consensus} problems. In this latter setting, the notion of spatially inhomogeneous evolutionary games has been recently proposed \cite{Ambrosio} (see also \cite{Mor1} for a related numerical scheme). There, the dynamics is not the outcome of an underlying non-local optimal control problem, but is determined by the agents' local (in time and space) decisions, as in the well-known replicator dynamics \cite{Josef}. 

We give an overview of the model in  \cite{Ambrosio}, which is relevant for the purpose of the paper. The position of an agent is described by $x \in \mathbb{R}^d$, while $U$ denotes the set of \emph{pure} strategies. A pay-off function $J\colon (\mathbb{R}^d\times U)^2\to \mathbb{R}$ is given,  so that $J(x,u,x',u')$ is the pay-off that {a} player in position~$x$ gets playing pure strategy $u$ against {a} player in position $x'$ with pure strategy $u'$.
However, agents are assumed to play different strategies according to a probability measure $\sigma \in \cP(U)$, which is referred to as {a} \emph{mixed} strategy. Hence, the state variable is given by the pair $(x,\sigma)$ accounting for the  position and the mixed strategy of an agent and 
\begin{equation*}
    \int_U J(x,u,x',u')\,\de \sigma'(u')
\end{equation*}
is the pay-off that {a} player in position $x$ gets playing strategy $u$ against {a} player in position $x'$ with mixed strategy $\sigma'$. If we then consider $N$ agents, whose states are denoted by $(x^i_t,\sigma^i_t)$, $i=1,\dots,N$, the pay-off that the $i$-th player gets playing strategy $u$ against all the other players at time $t$ is
\begin{equation*}
    \mathcal{J}(x_t^i,u)\coloneqq \frac{1}{N}\sum\limits_{j=1}^N \int_U J(x^i_t,u,x^j_t,u')\,\de\sigma_t^j(u').
\end{equation*}
In order to maximize this pay-off, the $i$-th player has to compare it with the mean pay-off over all possible strategies according to their mixed strategy $\sigma^i_t$. This leads us to the system of ODEs 
\begin{equation*}
    \begin{cases}
    \dot{x}_t^i = v(x^i_t,\sigma^i_t)  \\ 
    \dot{\sigma}_t^i = \left(\mathcal{J}(x_t^i,\cdot)-\displaystyle{\int_U} \mathcal{J}(x_t^i,v)\,\de\sigma_t^i(v)\right)\sigma_t^i
    \end{cases}\qquad \textrm{for }i=1,\dots,N,\,\,t\in (0,T].
\end{equation*}

In the later contribution \cite{MS2020}, the well-posedness theory as well as the mean-field approximation of the above system have been inserted in a more general framework which is suitable for a broader range of applications.
In this setting, the velocity~$v$ of each agent is also depending on the behavior of the other ones, and the replicator dynamics for the strategies has been replaced by a more general vector field $\cT$, that is 
\begin{equation}\label{eq:system_1}
\begin{cases}
\dot{x}_t^i = v_{\Lambda_t^N}(x^i_t,\sigma_t^i)\\[3mm]
 \dot{\sigma}_t^i = \cT_{\Lambda_t^N}(x_t^i,\sigma_t^i)
\end{cases}\qquad \textrm{for }i=1,\dots,N,\,\,t\in(0,T],
\end{equation}
where $\Lambda_t^N = \sum_{j=1}^N \delta(x_t^j,\sigma^j_t)\in \cP(\mathbb{R}^d\times \cP(U))$ is 
a distribution of agents with strategies at time~$t$. 
The interpretation, given in \cite{MS2020}, of these types of systems has a wider scope than the one of game theory: the interacting agents are assumed to belong to a number of different species, or populations, and therefore, more in general, we deal with \textit{labels} $\ell^i$ instead of (mixed) strategies~$\sigma^i$. 
This point of view can be used to distinguish informed agents steering pedestrians, to highlight the influence of few key investors in the stock market, or to recognize leaders from followers in opinion formation models. Throughout this work, we will adopt this perspective. 
Under a rather general set of assumptions on~$v$ and~$\cT$ (which, in particular, encompass the case of the replicator dynamics), it has been shown in \cite{MS2020} that the empirical measures $\Lambda_t^N$ associated with system \eqref{eq:system_1} converge to a probability measure on the state space, which solves the continuity equation
\begin{equation}\label{eq:cont}
\partial_t\Lambda_t + \textrm{div}(b_{\Lambda_t}\,\Lambda_t) = 0,
\end{equation}
where $b_{\Lambda_t}$ is the vector field which drives the state in system \eqref{eq:system_1}.

In \cite{entropy}, a further research direction has been explored. 
There, the replicator equation is slightly modified adding an \emph{entropy regularization} $\mathcal{H}$, see \eqref{entropy_functional_intro} below.
Besides providing a mean-field theory for such systems, the authors discuss the \emph{fast reaction limit} scenario, modeling situations in which the strategy (or label) switching of particles in the systems is actually happening at a faster time scale than that of the agents' dynamics.
This leads us to the purpose of our paper.

\smallskip

\noindent\textbf{Contribution of the present work.} In the present paper, we complement the abstract framework of \cite{MS2020} by adding an entropy regularization and we analyze its effects on the dynamics from an abstract point of view.
We fix a reference probability measure $\eta\in\cP(U)$ and we consider only diffuse probability densities~$\ell$ with respect to~$\eta$.
We set 
\begin{subequations}\label{entropy_functional_intro}
\begin{equation}
\cH(\ell) \coloneqq \ell \big[I(\ell)-\log(\ell)\big],
\end{equation}
where $I(\ell)$ is the negative entropy of the probability density $\ell$, namely
\begin{equation}
I(\ell)\coloneqq \int_{U} \ell(u)\,\log(\ell (u)) \,\de\eta(u).
\end{equation}
\end{subequations}
Then we analyze the system
\begin{equation}\label{one1}
\begin{cases}
\dot{x}_t^i = v_{\Lambda_t^N}(x^i_t,\ell_t^{\,i})\\[2mm]
\dot{\ell}_t^{\,i} = \lambda\,[\cT_{\Lambda_t^N}(x_t^i,\ell_t^{\,i})+\varepsilon\, \mathcal{H}(\ell_t^{\,i})]
\end{cases}\qquad i=1,\dots,N,\,\,t\in(0,T],
\end{equation}
where $\ell^{\,i}_t$ denotes the label of the $i$-th agent, $\varepsilon>0$ is a small parameter which modulates the intensity of the entropy functional, and $\lambda\geq 1$ takes into account the possible time scale difference between the positions and labels dynamics.
In the particular case where~$\cT_\Lambda$ is the operator of the replicator dynamics, this is exactly the system considered in \cite{entropy}.
The motivation for this regularization has already been discussed in \cite{entropy}: it serves to avoid degeneracy of the labels (see \cite[Example~2.1]{entropy} for a precise discussion) and allows for faster reactions to changes in the environment. We also refer to \cite{FC2005} for an earlier contribution on entropic regularizations in a game-theoretical setting.

From the mathematical point of view, the state space for the labels becomes now $\cP(U)\cap L^p(U,\eta)$ for some $p>1$.
As non-degeneracy is a desirable feature also for the wider setting considered in \cite{MS2020}, our first goal is then to establish a well-posedness theory in a similar spirit for system \eqref{one1}.
As it happened in \cite{MS2020}, a crucial point is giving a suitable set of assumptions on the dynamics which allows one to rely on the stability estimates for ODE's in convex subsets of Banach spaces developed in \cite[Section~I.3, Theorem~1.4, Corollary~1.1]{Brezis} and recalled in Theorem~\ref{Brezis_cor} below.
In particular, a sufficient set of assumptions on the operator~$\cT$ which complies with this setting is given at the beginning of Section~\ref{sec_decoupled}, see (T1)--(T3).
It slightly adapts and, to some extent, simplifies the assumptions on \cite{MS2020}, since here we are only considering the case of diffuse measures, and comprises both the case of the replicator dynamics and some models of leader-follower interactions with label switching modeled by reversible Markov chains \cite{Mor1} (see Remark~\ref{r:examples}).

The well-posedness of the particle model is proved in Theorem~\ref{Entropic_Problem_theorem} as a consequence of the estimates in Proposition~\ref{regularity_properties_entropic_vector_field}.
The convergence to a mean-field limit is discussed in the subsequent Section~\ref{ELS}.
In Section~\ref{Fast_Reaction_Limit}, instead, we focus on the special case of replicator-type models and revisit the results of \cite{entropy} from an abstract and more general point of view, which may also account for further modeling possibilities.

More precisely, we assume that the operator~$\cT$ takes the form
\begin{equation}
\label{fast_reaction_intro}
\cT_{\Lambda} (x, \ell) \coloneqq \left( \int_{U} \partial_\xi F_{\mu} ( x , \ell (u) , u ) \ell (u) \, \de \eta (u) - \partial_\xi F_{\mu } ( x , \ell , \cdot ) \right ) \ell,
\end{equation}
for $x \in \R^{d}$ and $\ell \in\cP(U)\cap L^{p}(U, \eta)$, and where $\mu$ is the marginal of $\Lambda$ in $\R^{d}$. In \eqref{fast_reaction_intro}, $\partial_\xi$ denotes the derivative of $F$ with respect to its second variable.

As we discuss in Remark~\ref{rem5.1}, for a proper choice of $F_\mu$, the above setting encompasses the case of \emph{undisclosed} replicator dynamics. By undisclosed it is meant that the players are not aware of their opponents' strategies. This is exactly the case dealt with in \cite{entropy}; see \cite[Remark~2.9]{entropy} for the difficulties connected to the fast reaction limit in the general case.
We stress, however, that~\eqref{fast_reaction_intro} has a more flexible structure than the case-study of the replicator dynamics. 
For instance, as we discuss again in Remark~\ref{rem5.1}, it allows one to consider pay-offs depending also on how often a strategy is played, penalizing choices that become predictable by other players. 
From the mathematical point of view, examples of functions fulfilling our hypotheses (F1)--(F5) of Section~\ref{Fast_Reaction_Limit} are discussed in Proposition~\ref{integralcase}.

For a system of the form \eqref{one1} with~$\cT$ given by \eqref{fast_reaction_intro}, we perform the fast reaction limit $\lambda\to +\infty$. 
This corresponds to a reasonable modeling assumption, that the label dynamics takes place at a much faster rate that the spatial dynamics.
In Theorem~\ref{Newton} we prove the convergence of system \eqref{one1}--\eqref{fast_reaction_intro} to a Newton-like system of the form 
\begin{equation*}
    \dot{x}_t^i = v_{\Lambda_t^N}(x_t^i,\ell_t^{*\,i}(x_t^1,\dots,x_t^N)),\qquad\textrm{for }i=1,\dots,N,\,\,t\in (0,T],
\end{equation*}
where $\ell_t^{*\,i}$ optimizes the functional
\begin{equation}\label{Gfun_intro}
G_{\mu} (x,\ell) \coloneqq  \int_{U} \big( F_{\mu}(x,\ell(u),u) + \varepsilon \ell(u) ( \log (\ell(u))  - 1) \big)\,\de\eta(u), \qquad \text{for $\ell \in C_{\varepsilon}$} 
\end{equation}
for fixed~$x$ and~$\mu$.
We stress that, differently from \cite{entropy}, we do not need to explicitly compute the minimizer as it was done in the special case of the replicator dynamics.
We remark that a crucial assumption for our proofs in Section~\ref{Fast_Reaction_Limit} is \emph{convexity} of the function~$F$ with respect to~$\ell$ and actually our proofs are guided by the heuristic intuition that, for fixed~$x$ and~$\mu$, the label equation in \eqref{one1}--\eqref{fast_reaction_intro} is the formal gradient flow of \eqref{Gfun_intro} with respect to the spherical Hellinger distance of probability measures \cite{KV2019} (see also \cite{Mor1}).
However, we provide explicit computations which do not resort to this gradient flow structure.

\smallskip

\noindent\textbf{Outlook.}
The present paper provides the well-posedness theory and the mean-field approximation for multi-population agent-based systems with an entropic regularization on the labels. 
We remark that such a regularization in the trajectories prevents concentration in the space of labels.
An analogous role could be played by diffusive terms in the space of positions, whose effects we plan to address in future contributions.
We also provide an abstract structure on the evolution of the labels to perform fast reaction limits, which in particular contains the special case of \cite{entropy}.
On the one hand, the assumption that one agent is not fully aware of the label distribution of the other ones (the so-called undisclosed setting we consider here) is realistic in many applications.
On the other hand, it would be interesting to single out the right assumptions to overcome this restriction while performing the fast reaction limite, for instance allowing one to consider~$F$ depending on the whole~$\Lambda$, and not only on the marginal~$\mu$, in \eqref{fast_reaction_intro}.

\smallskip

\noindent\textbf{Overview of the paper.}
In Section~\ref{sec:prel}, we present our notation, recall some tools of functional analysis and measure theory, and outline the basic settings of the problem. 
In Section~\ref{sec_decoupled}, we present the general assumptions and we study the entropic dynamical system \eqref{one1}, proving its well-posedness.
In Section~\ref{ELS}, we prove the mean-field limit of \eqref{one1} to a continuity equation such as \eqref{eq:cont}. 
In Section~\ref{Fast_Reaction_Limit}, we obtain the fast reaction limit of system \eqref{one1}, together with the explicit rate of convergence in terms of the parameter~$\lambda$.

\section{Preliminaries}\label{sec:prel}
\subsection{Basic notation}\label{BN}
If $(\mathcal{X},\mathsf{d}_\mathcal{X})$ is a metric space we denote by $\cP(\mathcal{X})$ the space of probability measures on $\mathcal{X}$. The notation $\cP_c(\mathcal{X})$ will be used for probability measures on  $\mathcal{X}$ having compact support. We denote by $C_0(\mathcal{X})$ the space of continuous functions vanishing at the boundary of~$\mathcal{X}$, and by $C_b(\mathcal{X})$ the space of bounded continuous functions. Whenever $\mathcal{X}=\mathbb{R}^d$, $d\geq 1$, it remains understood that it is endowed with the Euclidean norm (and induced distance), which shall be simply denoted by $\vert\cdot\vert$. For a Lipschitz function $f\colon \mathcal{X}\rightarrow \mathbb{R}$ we denote by 
\begin{equation*}
\mathrm{Lip}(f)\coloneqq \sup_{\substack{x,\,y\, \in\mathcal{X}  \\ x  \neq y}}\dfrac{\vert f(x)-f(y)\vert}{\mathsf{d}_\mathcal{X}(x,y)}
\end{equation*}
its Lipschitz constant. The notations $\mathrm{Lip}(\mathcal{X})$ and $\mathrm{Lip}_b (\mathcal{X})$ will be used for the spaces of Lipschitz and bounded Lipschitz function on $\mathcal{X}$, respectively.  Both are normed spaces with the norm $\norm{f} \coloneqq \norm{f}_\infty + \mathrm{Lip}(f)$, where $\norm{\cdot}_\infty$ is the supremum norm.  In a complete and separable metric space $(\mathcal{X},\mathsf{d}_\mathcal{X})$, we shall use the Kantorovich-Rubinstein distance $\mathcal{W}_1$ in the class of $\cP(\mathcal{X})$, defined as 
\begin{equation}\label{W1}
\mathcal{W}_1(\mu,\nu)\coloneqq \sup\left\{\,\int_{\mathcal{X}}\varphi(x)\,\de\mu(x)-\int_{\mathcal{X}}\varphi(x)\,\de\nu(x)\,\colon\varphi\in \mathrm{Lip}_b (\mathcal{X}),\,\mathrm{Lip}(\varphi)\leq 1 \right\}
\end{equation}
or, equivalently (thanks to the Kantorovich duality), as 
\begin{equation*}
\mathcal{W}_1(\mu,\nu)\coloneqq \inf\left\{\,\int_{\mathcal{X}\times\mathcal{X}}\mathsf{d}_\mathcal{X}(x,y)\,\de\Pi(x,y)\,\colon\Pi(A\times\mathcal{X})=\mu(A),\,\,\Pi(\mathcal{X}\times B)=\nu(B)\right\},
\end{equation*}
involving couplings $\Pi$ of $\mu$ and $\nu$. It can be proved that the infimum is actually attained. Notice that $\mathcal{W}_1(\mu,\nu)$ is finite if $\mu$ and $\nu$ belong to the space 
\begin{equation}\label{P1}
\cP_1(\mathcal{X}) \,\coloneqq \left\{\mu\in\cP(\mathcal{X})\colon \int_\mathcal{X} \mathsf{d}_\mathcal{X}(x,\overline{x})\,\de\mu(x)<+\infty \textrm{ for some } \overline{x}\in\mathcal{X} \right\}
\end{equation}
and that $(\cP_1(\mathcal{X}),\mathcal{W}_1)$ is complete if $(\mathcal{X},\mathsf{d}_\mathcal{X})$ is complete. For a probability measure $\mu\in\cP(\mathcal{X})$, if $\mathcal{X}$ is also a Banach space, we define the first moment $m_1(\mathcal{\mu})$ as 
\begin{equation*}
m_1(\mu)\coloneqq \int_\mathcal{X} \norm{x}_{\mathcal{X}}\,\de\mu(x).
\end{equation*}
So that, the finiteness of the integral above is equivalent to $\mu\in\cP_1(\mathcal{X})$, whenever the distance $\mathsf{d}_\mathcal{X}$ is induced by the norm $\norm{\cdot}_{\mathcal{X}}$\,.

Let $\mu\in\cP(\mathcal{X})$ and $f\colon\mathcal{X}\rightarrow Z$ a $\mu$-measurable function be given. The push-forward measure $f_\# \mu\in \cP(Z)$ is defined by $f_\# \mu (B) = \mu(f^{-1}(B))$ for any Borel set $B\subset Z$.  It also holds the change of variables formula
\begin{equation*}
\int_Z g(z)\,\de f_\#\mu(z) = \int_\mathcal{X} g(f(x))\,\de\mu(x)
\end{equation*} 
whenever either one of the integrals is well defined.

For $E$ being a Banach space, the notation $C^1_b(E)$ will be used to denote the subspace $C_b(E)$ of functions having bounded continuous Fréchet differential at each point. The notation $D\phi(\cdot)$ will be used to denote the Fréchet differential. In the case of a function $\phi\colon [0,\,T]\times E \rightarrow \mathbb{R}$, the symbol $\partial_t$ will be used to denote partial differentiation with respect to $t$, while $D$ will only stand for the differentiation with respect to the variables in $E$.

\subsection{Functional setting}\label{FS}
The space of labels $(U,\mathsf{d})$ will be assumed to be a 
compact metric space. Consider the Borel $\sigma$-algebra $\mathfrak{B}$ on $U$ induced by the metric $\mathsf{d}$ and let us fix a probability measure $\eta \in \cP(U)$ which we can assume, without loss of generality, to have full support, \emph{i.e.}, $\mathrm{spt}(\eta)=U$. Notice that the measure space $(U,\mathfrak{B},\eta)$ is $\sigma$-finite and separable.
For $p\in[1,+\infty]$, we 
consider the space $L^p(U,\eta)$, which is a separable Banach space. 
Given $r$ and $R$ such that $0\leq r <1< R \leq + \infty$, we introduce the set of probability densities with respect to~$\eta$, having lower bound $r$ and upper bound $R$:
\begin{equation}\label{Ceps}
C_{r,R} \coloneqq \left\{\ell\in L^p(U,\eta): \int_{U} \ell(u)\,\de\eta(u) = 1 \mathrm{\,\,and\,\,} r\leq \ell\leq R\,\, \eta\textrm{-}a.e. \right\};
\end{equation}
notice that $C_{0,\infty}$ is the set of $L^p$-regular probability densities with respect to $\eta$.
Since $\eta(U)=1$, the inclusion $L^p(U,\eta)\subset L^1(U,\eta)$ holds for all $p\in [1,+\infty]$ and therefore the sets $C_{r,R}$ are closed with respect to the $L^p$-norm. Thus, when equipped with the $L^p$-norm, the sets $C_{r,R}$ are separable\footnote{\,A subset of a separable metric space is also separable, \cite{Folland}.}. 
Finally, notice that $C_{r,R}$  are also convex and their interiors are empty. 

\smallskip

The state variable of our system is $y\coloneqq (x,\ell)\in \mathbb{R}^d\times C_{0,\infty} \eqqcolon Y$. The component $x\in \mathbb{R}^d $ describes the location of an agent in space, whereas the component $\ell\in C_{0,\infty}$ describes the distribution of labels 
of the agent. A probability distribution $\Psi \in \cP(Y)$ denotes a distribution of agents with labels. 
To outline the functional setting for the dynamics, we define $\overline{Y} \coloneqq \mathbb{R}^d\times L^p(U,\eta)$ and the norm  $\norm{\cdot}_{\overline{Y}}$ by 
\begin{equation}\label{norm}
\norm{y}_{\overline{Y}} = \norm{(x,\ell)}_{\overline{Y}} \coloneqq |x|+\norm{\ell}_{L^p(U,\eta)}.
\end{equation}
Since 
$Y\subset \overline{Y}$, we equip 
$Y$ with the $\norm{\cdot}_{\overline{Y}}$ norm. For a given $\varrho>0$, we denote by $B_\varrho$ the closed ball of radius $\varrho$ in $\mathbb{R}^d$ and by 
$B_\varrho^{Y}$ the closed ball of radius $\varrho$ in 
$Y$, namely, 
$B_\varrho^{Y} = \{y\in\ Y : \norm{y}_{\overline{Y}}\leq \varrho\}$.
The Banach space structure of $\overline{Y}$ allows us to define the first moment $m_1(\Psi)$ for a probability measure 
$\Psi\in\cP(Y)$ as
\begin{equation*}
m_1(\Psi)\coloneqq 
\int_{Y}\norm{y}_{\overline{Y}}\,\de\Psi(y),
\end{equation*}
so that the space 
$\cP_1(Y)$ defined in \eqref{P1}
can be equivalently characterized as 
\begin{equation*}
\cP_1(Y) = \{\Psi\in\cP(Y) : m_1(\Psi)<+\infty\}.
\end{equation*}

Whenever we fix $r$ and $R$ in \eqref{Ceps}, we set $Y_{r,R}\coloneqq \R^{d}\times C_{r,R}$ and we modify the notation above accordingly.

\medskip

We conclude this section by recalling the following existence result for ODEs of convex subsets of Banach spaces, which is stated in \cite[Corollary~2.3]{MS2020} and \cite[Theorem~1]{AAMS2022}, generalizing the well-known results of \cite[Section~I.3, Theorem~1.4, Corollary~1.1]{Brezis}.
\begin{theorem}
\label{Brezis_cor}
Let $(E,\norm{\cdot}_E)$ be a Banach space, let $C$ be a closed convex subset of $E$, and, for $t\in [0,T]$, let $A(t,\cdot)\colon C \rightarrow E$ be a family of operators satisfying the following properties:
\begin{itemize}
\item[(i)]for every $\varrho>0$ there exists a constant $L_\varrho > 0$ such that for every $t\in [0,T]$ and $c_1$, $c_2\in C\cap\{e\in E : \norm{e}_E\leq \varrho\}$
\begin{equation*}
\norm{A(t,c_1)-A(t,c_2)}_E\leq L_\varrho \norm{c_1-c_2}_E;
\end{equation*}
\item[(ii)] for every $c\in C$ the map $t\mapsto A(t,c)$ belongs to $L^1([0,T];E)$;
\item[(iii)] for every $\varrho>0$ there exists $\theta_\varrho > 0$ such that for every $c \in C\cap\{e\in E : \norm{e}_E\leq \varrho\}$
\begin{equation*}
c+ \theta_\varrho A(t,c) \in C;
\end{equation*}
\item[(iv)] there exists $M>0$ such that for every $c\in C$, there holds
\begin{equation*}
\norm{A(t,c)}_E \leq M(1+\norm{c}_E).
\end{equation*}
\end{itemize}
Then for every $\overline{c}\in C$ there exists a unique curve $c\colon [0,T]\rightarrow C$ of class $C^1$ such that
\begin{equation}\label{Brezis_ODE}
\frac{\de}{\de t} c_t= A(t,c_t)\quad \textrm{in }[0,T], \quad c_0 = \overline{c}.
\end{equation}
Moreover, if $c^1,c^2$ are the solutions with initial data $\overline{c}^1,\overline{c}^2\in C\cap \{e\in E\colon \norm{e}_E\leq \varrho\}$, respectively, there exists a constant $L=L(M,\varrho,T)>0$ such that 
\begin{equation}\label{Brezis_Estimate}
\norm{c^1_t-c^2_t}_E \leq e^{Lt}\,\norm{\overline{c}^1-\overline{c}^2}_E\qquad \textrm{for every } t\in [0,T].
\end{equation}
\end{theorem}

\section{Well-posedness of the entropic system}
\label{sec_decoupled}

In this section, we study the well-posedness of the $\varepsilon$-regularized entropic system \eqref{one1}; for convenience, in this section, we fix $\lambda=1$. 
We start by listing the assumptions on the velocity field $y\mapsto v_\Psi(y)$ and on the transfer map $y\mapsto \cR_\Psi^\varepsilon(y)\coloneqq \cT_\Psi(y)+\varepsilon\cH(\ell)$.
We assume that the velocity field $v_\Psi \colon Y\rightarrow \mathbb{R}^d$ satisfies the following conditions:
\begin{itemize}
\item[(v1)] for every $\varrho>0$, for every $\Psi\in\cP(B_\varrho^{Y})$, $v_\Psi\in \textrm{Lip}(B_\varrho^Y;\mathbb{R}^d)$ uniformly with respect to $\Psi$, namely there exists $L_{v,\varrho}>0$ such that
\begin{equation*}
\vert v_\Psi(y^1)-v_\Psi(y^2)\vert \leq L_{v,\varrho} \vert\vert y^1-y^2\vert\vert_{\overline{Y}}\,;
\end{equation*}
\item[(v2)] for every $\varrho>0$, there exists $L_{v,\varrho}>0$ such that for every $y\in B_\varrho^Y$\,, and for every $\Psi^1$, $\Psi^2\in\cP(B_\varrho^Y)$
\begin{equation*}
\vert v_{\Psi^1}(y)-v_{\Psi^2}(y)\vert\leq L_{v,\varrho} \mathcal{W}_1(\Psi^1,\Psi^2);
\end{equation*}
\item[(v3)] there exists $M_v>0$ such that for every $y\in Y$, and for every $\Psi\in \cP_1(Y)$ there holds
\begin{equation*}
\vert v_\Psi(y)\vert \leq M_v (1+\norm{y}_{\overline{Y}}+m_1(\Psi)).
\end{equation*}
\end{itemize}
We now describe the assumptions on $\cT$. For every $\Psi\in\cP_1(Y)$, let $\cT_\Psi\colon Y \to L^p(U,\eta)$ be an operator such that
\begin{itemize}
\item[(T1)] $\cT_\Psi(y)$ has zero mean for every  $(y,\Psi)\in Y\times \cP_1(Y) $\,: 
\begin{equation*}
\int_U \cT_\Psi(y)(u)\,\de\eta(u) = 0\,;
\end{equation*}
\item[(T2)] for every $\varrho>0$ there exists $L_{\cT,\varrho}>0$ such that for every $(y^1,\Psi^1)$, $(y^2,\Psi^2)\in  B_\varrho^Y\times \cP(B_\varrho^Y)$
\begin{equation*}
\vert\vert\cT_{\Psi^1}(y^1)-\cT_{\Psi^2}(y^2)\vert\vert_{L^p(U,\eta)} \leq L_{\cT,\varrho} \big(\vert\vert y^1-y^2\vert\vert_{\overline{Y}}+\mathcal{W}_1(\Psi^1,\Psi^2)\big);
\end{equation*}
\item[(T3)] there exist a monotone increasing function $\omega\colon [0,+\infty)\to [0,+\infty)$\,, for which 
\begin{equation*}
\limsup_{s\rightarrow 0^+} \frac{\omega(s)}{s} \eqqcolon \underline{\omega}\in [0,+\infty) \qquad\textrm{and}\qquad \limsup_{s\rightarrow \infty} \frac{\omega(s)}{s} \eqqcolon \overline{\omega}\in [0,+\infty),
\end{equation*} 
and a constant $C_\cT>0$ such that for every $(y,\Psi)\in Y_{r,R}\times \cP_1(Y)$ (for some $0<r<1<R<+\infty$), 
\begin{equation*}
\cT_\Psi(y)(u)  \leq C_\cT \omega(R) \qquad\text{and}\qquad
(\cT_\Psi(y)(u))_-  \leq C_\cT \omega(\ell(u)),
\end{equation*}
for $\eta$-almost every $u\in U$.
\end{itemize}
Finally, the entropy functional $\cH\colon C_{0,\infty}\to L^0(U,\eta)$ that we consider is defined by 
\begin{equation*}
\cH(\ell) \coloneqq \ell \big[I(\ell)-\log(\ell)\big],
\end{equation*}
where $I(\ell)$ is the negative entropy of the probability density $\ell$, namely
\begin{equation*}
I(\ell)\coloneqq \int_{U} \ell(u)\,\log(\ell (u)) \,\de\eta(u).
\end{equation*}
We notice that, for every $r,R\in(0,+\infty)$ and every $\ell\in C_{r,R}$, we have that $\cH(\ell)\in L^p(U,\eta)$ for every $p\in[1,+\infty]$.

\begin{remark}
\label{r:examples}
We remark that assumptions~$\mathrm{(v1)}$--$\mathrm{(v3)}$ already appeared in~\cite{AAMS2022, Mor1, MS2020} and in~\cite{Ambrosio, entropy} in a stronger form and are rather typical in the study of ODE systems. Conditions $\mathrm{(T1)}$--$\mathrm{(T3)}$, instead, are slightly different from the usual hypotheses on the operator $\mathcal{T}_{\Psi}$ introduced in~\cite[Section~3]{MS2020}. In particular, $\mathrm{(T3)}$ involves a pointwise condition on~$\mathcal{T}_{\Psi} (y)$, which is crucial to show existence and uniqueness of solutions to the $N$-particles system~\eqref{coup_entropic_problem_compact} below. The role played by such assumption is that of guaranteeing a pointwise control on the strategy~$\ell(u)$, ensuring a bound from above and from below away from~$0$. For more details, we refer to the proof of Proposition~\ref{regularity_properties_entropic_vector_field}.

Here, we report two fundamental examples that fall into our theoretical framework. The first one is the {\em replicator dynamics} (see also~\cite{Ambrosio, entropy}). If $\Psi \in \mathcal{P}( Y)$ stands for the distribution of players with mixed strategies $\ell' \in C_{0, \infty}$, the pay-off that a player in position~$x$ gets playing the strategy~$u\in U$ against all the other players writes
\begin{equation}
\label{e:JLambda}
\mathcal{J}_{\Psi}(x,u) = \int_{Y}\int_U J(x,u,x',u')\,\ell'(u')\,\de\eta(u') \, \de\Psi (x',\ell')
\end{equation}
and the corresponding operator~$\mathcal{T}$ is
\begin{equation*}
\cT_\Psi (x,\ell) = \left(\mathcal{J}_{\Psi} (x,\cdot) - \int_U \mathcal{J}_{\Psi} (x,u) \ell(u) \,\de\eta(u) \right)\ell\,.
\end{equation*}
In \cite[Proposition~5.8]{MS2020} sufficient conditions on~$J$ are provided, that imply conditions~$\mathrm{(T1)}$ and~$\mathrm{(T2)}$. If $J$ is bounded in $\mathbb{R}^{d} \times U \times \mathbb{R}^{d} \times U$, then $\mathcal{T}$ also satisfies~$\mathrm{(T3)}$.

The second example stems from population dynamics and models a leader-follower interactions (see~\cite[Sections 4 and 5]{MS2020}). We assume that $U= \{ 1, \ldots, H\}$ for some~$H \in \mathbb{N}$ denotes the set of possible labels within a population. Given a distribution~$\Psi \in \mathcal{P}(Y)$ of agents with labels $\ell \in L^{p} (U, \eta)$, for $h \neq k \in U$ we denote by $\alpha_{hk} (x, \Psi) \geq 0$ the rate of change from label $h$ to label~$k$ and set
\begin{equation}
\label{e:alphahh}
\alpha_{hh}(x, \Psi) \coloneqq \sum_{k \neq h} \alpha_{kh} (x, \Psi) \,.
\end{equation}
Since $\eta$ is supported on the whole of~$U$, we may identify $\ell \in L^{p} (U, \eta)$ with the vector~$(\ell_{1}, \ldots, \ell_{H})$. Hence, the operator~$\mathcal{T}_{\Psi}$ is defined by
\begin{displaymath}
(\mathcal{T}_{\Psi} (y))_{h} \coloneqq (\mathcal{Q}^{*} (x, \Psi) \ell)_{h} = -\alpha_{hh} (x, \Psi) \ell_{h} + \sum_{k \neq h} \alpha_{kh} (x, \Psi) \ell_{k}\,, 
\end{displaymath}
where the matrix~$\mathcal{Q}(x, \Psi)$ writes as
\begin{displaymath}
\mathcal{Q}(x, \Psi) \coloneqq \left( 
\begin{array}{cccc}
-\alpha_{11} (x, \Psi) & \alpha_{12} (x, \Psi) & \cdots & \alpha_{1H} (x, \Psi) \\
\alpha_{21} (x, \Psi) & -\alpha_{22} (x, \Psi) & \cdots & \alpha_{2H} (x, \Psi) \\
\vdots & \vdots & \ddots &\vdots \\
\alpha_{H1} (x, \Psi) & \alpha_{H2} (x, \Psi) & \cdots & -\alpha_{HH} (x, \Psi)
\end{array}\right).
\end{displaymath}
Suitable assumptions on~$\alpha_{kh}$ that ensure~$\mathrm{(T1)}$ and~$\mathrm{(T2)}$ are given in~\cite[Proposition~5.1]{MS2020}. Once again, if $\alpha_{kh}$ are bounded, we have~$\mathrm{(T3)}$ as well thanks to the precise structure~\eqref{e:alphahh}: in particular, the positivity of~$\alpha_{kh}$ for every $k\neq h$ is crucial to estimate $\big(\cT_{\Psi}(y)(u)\big)_-$ in terms of the sole $\ell(u)$.
\end{remark}

\begin{proposition}
\label{regularity_properties_entropic_vector_field}
Assume that $v_\Psi\colon Y \to \mathbb{R}^d$ satisfies $\mathrm{(v1)}$--$\mathrm{(v3)}$ and $\cT_\Psi\colon Y \to L^p(U,\eta)$ satisfies $\mathrm{(T1)}$--$\mathrm{(T3)}$. Then, for every $\varepsilon>0$ there exist $r_\varepsilon\in (0,1)$ and $R_\varepsilon\in(1,+\infty)$ such that -- setting $Y_\varepsilon\coloneqq Y_{r_\varepsilon,R_\varepsilon}$ -- for every $\Psi\in \cP_1(Y_\varepsilon)$, 
the vector field $b^\varepsilon_\Psi\colon Y_\varepsilon \rightarrow \overline{Y}$ defined as
\begin{equation}\label{entropic_vector_field}
b^\varepsilon_\Psi(y)\coloneqq \begin{pmatrix}
v_\Psi(y)\\ \cR^\varepsilon_\Psi(y)
\end{pmatrix}, \qquad \text{for every $y \in Y_\varepsilon$,}
\end{equation} 
satisfies the following properties:

\begin{enumerate}

\item for every $\varrho>0$, there exists $L_{\varepsilon, \varrho}>0$ such that for every $\Psi\in\cP(B_\varrho^{Y_\varepsilon})$, and for every $y^1,y^2\in B_\varrho^{Y_\varepsilon}$ 
\begin{equation}\label{coup_help2}
\vert\vert b_\Psi^\varepsilon(y^1)-b_\Psi^\varepsilon(y^2)\vert\vert_{\overline{Y}}\leq L_{\varepsilon,\varrho} \vert\vert y^1 - y^2 \vert\vert_{\overline{Y}}\,;
\end{equation}

\item for every $\varrho>0$, there exists $L_\varrho>0$ such that for every $\Psi^1,\Psi^2\in\cP(B_\varrho^{Y_\varepsilon})$, and for every $y\in B_\varrho^{Y_\varepsilon}$
\begin{equation}\label{coup_help3}
\vert\vert b_{\Psi^1}^\varepsilon(y) - b_{\Psi^2}^\varepsilon(y)\vert\vert_{\overline{Y}}\leq L_\varrho \mathcal{W}_1(\Psi^1,\Psi^2)\,;
\end{equation}

\item there exists $M_\varepsilon>0$ such that for every $y\in {Y_\varepsilon}$ and for every $\Psi\in\cP_1({Y_\varepsilon})$ 
there holds
\begin{equation}\label{coup_help4}
\vert\vert b_\Psi^\varepsilon(y) \vert\vert_{\overline{Y}}\leq M_\varepsilon\,(1+\norm{y}_{\overline{Y}} + m_1(\Psi))\,.
\end{equation}

\item 
there exists $\theta_{\varepsilon} >0$ such that for every $\varrho>0$ and for every $y\in B_\varrho^{Y_\varepsilon}$ and for every $\Psi\in \cP(B_\varrho^{Y_\varepsilon})$
\begin{equation}\label{coup_help1}
y+\theta_{\varepsilon} b^\varepsilon_\Psi(y)\in Y_\varepsilon\,.
\end{equation}

\end{enumerate} 
\end{proposition}

\begin{proof}
The proof is divided into three steps.\\ 
\noindent\emph{Step 1 (boundedness of $\cH$).}
We start by proving that $\mathcal{H}(C_{r, R})\subset L^\infty(U,\eta)$ for every $r, R \in (0, +\infty)$ with $r < 1 < R$, which in turn implies that for every $\varrho \in (0, +\infty)$, every $\Psi \in \mathcal{P} (B_{\varrho} ^{Y_{r, R}})$, and every $y \in Y_{r, R}$,  $\cR^\varepsilon_{\Psi} (y)$ is well defined in $L^{p}(U, \eta)$. 

For every $u\in U$ we may write $ \ell(u) = r \zeta(u) + R (1-\zeta(u))$, with $0\leq \zeta(u) \leq 1$\,. Thus, using the convexity of the function $t\mapsto t\log(t)$ in $(0, +\infty)$ we get 
\begin{equation*}
I(\ell) \leq r  \log(r ) \int_{U} \zeta(u)\,\de\eta(u)+ R \log(R)\,\int_{U}(1-\zeta(u))\,\de\eta(u)\,.
\end{equation*}
Since $\ell$ is a probability density it is straightforward to check that 
\begin{equation*}
\int_U \zeta(u) \, \de\eta(u) = \frac{R - 1}{R - r}\,.
\end{equation*} 
Therefore, 
\begin{equation}
\label{e:Iell}
I(\ell) \leq \frac{R - 1}{R - r}\,r  \log(r) + \left (1 - \frac{R - 1}{R - r}\right) R \log(R).
\end{equation}
To simplify the notation, we define
\begin{equation}\label{alpha_eps}
\alpha_{r, R} \coloneqq \frac{(R - 1) r}{R - r}\in (0,1)\,,
\end{equation}
so that inequality~\eqref{e:Iell} reads
\begin{equation}\label{entropy_upper_bound}
I(\ell) \leq \alpha_{r, R} \log(r) + (1 - \alpha_{r, R}) \log(R) \eqqcolon k_{r, R}\,.
\end{equation}
Moreover, by Jensen's inequality we have that
\begin{equation}\label{entropy_lower_bound}
I(\ell) \geq \int_{U} \ell(u)\,\de\eta(u)\,\log \left(\int_{U} \ell(u) \, \de\eta(u)\right) = 0\,.
\end{equation}
Since $\ell \in C_{r, R}$ and~\eqref{entropy_upper_bound} and~\eqref{entropy_lower_bound} hold, we deduce that
\begin{equation}
\label{e:2-bounds-H}
-R  \log(R) \leq \mathcal{H}(\ell) \leq R \,k_{r, R} + \frac{1}{e}\,,
\end{equation}
so that $\mathcal{H}(\ell)\in L^\infty(U,\eta)$. 

Since~$\mathcal{H}(\ell)$ has zero mean and~$({\rm T}1)$ holds true, we have that
\begin{equation}
\label{e:R-mean}
\int_{U} \mathcal{R}_{\Psi}^{\varepsilon}(y)(u) \, \de \eta(u) = 0\,.
\end{equation}

\noindent\emph{Step 2 (Lipschitz continuity of $\cH$).} 
We now show that $\mathcal{H}$ is Lipschitz continuous on~$C_{r, R}$ with Lipschitz constant~$L_{r, R}$ depending on~$r$ and~$R$. Since $t\mapsto t\log(t)$ is Lipschitz continuous on $[r , R]$ whenever $r > 0$ (we let~$L_{r, R}'$ be its Lipschitz constant), we may estimate for every $\ell_{1}, \ell_{2} \in C_{r, R}$ and every $u \in U$
\begin{equation*}
\begin{split}
\vphantom{\int} \vert\mathcal{H}(\ell_1)(u) & - \mathcal{H}(\ell_2)(u) \vert \leq | I ( \ell_1) \ell_1(u) - I(\ell_2)\ell_2(u) | + |\ell_1(u)\log (\ell_1(u)) - \ell_2(u) \log(\ell_2(u)) | 
\\ 
&
\vphantom{\int} \leq | I(\ell_1) - I(\ell_2) | \, |\ell_1(u)| + | I(\ell_2)| \, |\ell_1(u) - \ell_2(u)| + L_{r, R}' |\ell_1(u) - \ell_2(u)| 
\\ 
&
\vphantom{\int} \leq  R  | I(\ell_1) - I(\ell_2) | + k_{r, R} |\ell_1(u) - \ell_2(u) | + L_{r, R}'  |\ell_1(u) - \ell_2(u) | 
\\ 
&
\leq R \int_U | \ell_1(u) \log(\ell_1(u)) - \ell_2(u) \log(\ell_2(u)) | \, \de \eta(u) + (k_{r, R} + L_{r, R}' ) |\ell_1(u) - \ell_2(u)| 
\\ 
&
\leq R L_{r, R}' \int_U  | \ell_1(u) - \ell_2(u) |\, \de \eta(u) + ( k_{r,R} + L_{r, R}' ) |\ell_1(u) - \ell_2(u)|\,.
\end{split}
\end{equation*}
Thus, there holds
\begin{equation}\label{e:Lipschitz-H}
\begin{split}
\norm{\mathcal{H}(\ell_1)-\mathcal{H}(\ell_2)}_{L^{p}(U, \eta)} & \leq R L_{r, R}'  \norm{\ell_1 - \ell_2}_{L^1(U,\eta)} + (k_{r, R} + L_{r, R}' )\,\norm{\ell_1-\ell_2}_{L^{p}(u, \eta)} 
\\ 
&
\leq  R L_{r, R}'  \norm{\ell_1 - \ell_2}_{L^{p}(U, \eta)} + (k_{r, R} + L_{r, R}' ) \norm{\ell_1 - \ell_2}_{L^{p}(U, \eta) } 
\\ 
&
= ((R+1) L_{r, R}' + k_{r, R} ) \norm{\ell_1 - \ell_2}_{L^{p}(U, \eta)} =: L_{r, R}  \norm{\ell_1 - \ell_2}_{L^{p}(U, \eta)}\,,
\end{split}
\end{equation}
where we have used that $\eta \in \mathcal{P}(U)$. 

\noindent\emph{Step 3 (proof of properties (1)--(4)).}
For $\varepsilon>0$, 
we fix~$r_\varepsilon \in (0, 1)$ such that
\begin{equation}\label{r_eps}
\varepsilon\,\log\left(\dfrac{3}{4\,r_\varepsilon}\right)\geq C_\cT\,\dfrac{\omega(\frac{4}{3}r_\varepsilon)}{r_\varepsilon} \,.
\end{equation}
Notice that, thanks to~$({\rm T}3)$, such~$r_{\varepsilon}$ exists as
\begin{equation*}
\limsup_{r \rightarrow 0^+} \,\varepsilon \log \left(\dfrac{3}{4 r} \right ) = +\infty \qquad \textrm{and}\qquad\limsup\limits_{r \rightarrow 0^+} \, C_\cT\, \dfrac{\omega(\frac{4}{3}r)}{r} = \frac{4}{3}\,C_\cT\,\underline{\omega} \,.
\end{equation*}
We now fix~$R_\varepsilon \in (1, +\infty)$ such that
\begin{equation}\label{R_eps}
\alpha_{r_{\varepsilon}, R_{\varepsilon}} \log\left(\frac{R_\varepsilon}{r_\varepsilon}\right) \geq \dfrac{2\,C_\cT\,\omega(R_\varepsilon)}{\varepsilon\,R_\varepsilon}.
\end{equation}
Again, notice that there exists at least one $R_\varepsilon>1$ satisfying~\eqref{R_eps} since, by~$({\rm T}3)$ and by definition of~$\alpha_{r, R}$ in~\eqref{alpha_eps}, it holds
\begin{equation*}
\limsup\limits_{R \rightarrow +\infty}\,\alpha_{r_{\varepsilon}, R} \log\left(\frac{R}{r_\varepsilon}\right) = +\infty \qquad \textrm{and}\qquad\limsup\limits_{R \rightarrow +\infty} \, \dfrac{2\,C_\cT\,\omega(R)}{\varepsilon\,R} = \dfrac{2\,C_\cT\,\overline{\omega}}{\varepsilon}\,.
\end{equation*}

For $r_{\eps}$ and $R_{\eps}$ given above, we now prove properties $(1)$--$(4)$. For simplicity, we set from now on $ C_{\varepsilon} \coloneqq C_{r_{\varepsilon}, R_{\varepsilon}}$, $Y_{\varepsilon} \coloneqq Y_{r_{\varepsilon}, R_{\varepsilon}}$, $\alpha_{\varepsilon} \coloneqq \alpha_{r_{\varepsilon}, R_{\varepsilon}}$, and $k_{\varepsilon} \coloneqq k_{r_{\varepsilon}, R_{\varepsilon}}$.

{\em Property (1).} Let $\varrho>0$, $\Psi \in \mathcal{P}(B_\varrho^{Y_{\varepsilon}})$, and $y^{1}, y^{2} \in B^{Y_{\varepsilon}}_{\varrho}$. By~$({\rm T}2)$, the operator~$\mathcal{T}_{\Psi}$ is Lipschitz continuous on~$B_{\varrho}^{Y_{\varepsilon}}$ with Lipschitz constant~$L_{\mathcal{T}, \varrho}>0$, while by~$({\rm v}1)$, $v_{\Psi}$ is Lipschitz continuous on~$B_{\varrho}^{Y_{\varepsilon}}$ with Lipschitz constant~$L_{v, \varrho}>0$. In view of the Lipschitz continuity of~$\mathcal{H}$ (cf.~\eqref{e:Lipschitz-H}), setting, for instance, $L_{\varepsilon, \rho}:= L_{v, \varrho} +  \max\{\varepsilon L_{r_{\varepsilon}, R_{\varepsilon}}, L_{\cT, \rho}\}$, we deduce~\eqref{coup_help2}.

{\em Property (2).} It is straightforward from~$({\rm v}2)$ and~$({\rm T}2)$, since the entropy regularization~$\mathcal{H}$ does not depend on~$\Psi \in \mathcal{P}(B^{Y_{\varepsilon}}_{\varrho})$.

{\em Property (3).} In view of~$({\rm v}3)$, it is enough to prove that there exists~$M_{\varepsilon}>0$ such that for every $y \in Y_{\varepsilon}$ and every $\Psi \in \mathcal{P}_{1}(Y_{\varepsilon})$
\begin{equation}
\label{e:R-sublinear}
\| \mathcal{R}^{\varepsilon}_{\Psi} (y)\|_{L^{p}(U, \eta)} \leq M_{\varepsilon} \big( 1 + \| y \|_{\overline{Y}} + m_{1} (\Psi) \big)\,.
\end{equation}
By~$({\rm T}3)$, we have that $| \mathcal{T} (y) (u) | \leq C_{\mathcal{T}} \omega (R_{\varepsilon})$. Recalling~\eqref{e:2-bounds-H} and setting 
$$
M_{\varepsilon} := C_{\mathcal{T}} \omega (R_{\varepsilon}) + \varepsilon \, \max \Big\{ R_{\varepsilon} \log R_{\varepsilon} \, , \,  R_{\varepsilon} \,k_{\varepsilon} + \frac{1}{e}\Big\}\,,
$$
we infer~\eqref{e:R-sublinear} and therefore~\eqref{coup_help4}.

{\em Property (4).} Let $\varrho>0$. Since $Y_{\varepsilon} = \R^{d} \times C_{\varepsilon}$, we only have to find~$\theta_\varepsilon$ such that for every $\Psi \in \mathcal{P}(B^{Y_{\varepsilon}}_{\varrho})$ and every $y = (x, \ell) \in B^{Y_{\varepsilon}}_{\varrho}$, 
\begin{equation}
\label{e:T4-R}
\ell + \theta_{\varepsilon} \mathcal{R}_{\Psi}^{\varepsilon} (x, \ell) \in C_{\varepsilon}\,.
\end{equation}  
In view of~\eqref{e:R-mean}, we already know that for any $\theta_{\varepsilon} > 0$
\begin{equation}
\label{e:mean=1}
\int_{U} \ell(u) + \theta_{\varepsilon} \mathcal{R}_{\Psi}^{\varepsilon} (x, \ell) (u) \, \de \eta(u) = 1\,.
\end{equation}
Hence, we have to show that upper and lower bounds of~$C_{\varepsilon}$ are preserved for a suitable choice of~$\theta_{\varepsilon}$ independent of~$y \in B^{Y_{\varepsilon}}_{\varrho}$ and of $\Psi \in \mathcal{P}(B_{\varrho}^{Y_{\varepsilon}})$. The precise~$\theta_{\varepsilon}$ will be specified along the proof. 

Let $y \in B^{Y_{\varepsilon}}_{\varrho}$ and $\Psi \in \cP(B_\varrho^{Y_\varepsilon})$. We start by imposing that for $\eta$-a.e.~$u\in U$ 
\begin{equation}
\label{e:to-show-1}
\ell(u) + \theta_\varepsilon\cR^\varepsilon(\ell)(u) \leq R_\varepsilon\,.
\end{equation}
Using $({\rm T}3)$ and~\eqref{entropy_upper_bound} we get that
\begin{equation}
\label{decoup_help1}
\begin{split}
\ell(u) & + \theta_\varepsilon\,\left[\mathcal{T}_{\Psi}(y)(u) + \varepsilon\,\mathcal{H}(\ell)(u)\right] \leq \ell(u) + \theta_\varepsilon\,\left[C_\cT\,\omega(R_{\varepsilon} ) + \varepsilon \mathcal{H} (\ell)(u) \right] 
\\ 
&
= \ell(u) + \theta_\varepsilon \left[C_\cT\,\omega(R_\varepsilon) + \varepsilon \ell(u) \left(I(\ell) - \log(\ell(u)) \right) \right] 
\\ 
&
\leq \ell(u) + \theta_\varepsilon \left[C_\cT\,\omega(R_\varepsilon) + \varepsilon\,\ell(u)\,(\alpha_\varepsilon\,\log(r_\varepsilon)+(1-\alpha_\varepsilon)\log(R_\varepsilon)-\log(\ell(u)))\right].
\end{split}
\end{equation}
Because of \eqref{R_eps} we have that
\begin{equation}
\label{e:limit_R}
\begin{split}
& \lim_{t\nearrow R_\varepsilon} [C_\cT \omega(R_\varepsilon)+\varepsilon t (\alpha_\varepsilon \log(r_\varepsilon)+(1-\alpha_\varepsilon)\log(R_\varepsilon)-\log t)] \\ 
=&\, C_\cT \omega(R_\varepsilon) - \varepsilon \alpha_\varepsilon R_\varepsilon \log\bigg(\frac{R_\varepsilon}{r_\varepsilon}\bigg) \leq -C_\cT\omega(R_\varepsilon) <0.
\end{split}
\end{equation}
Inequalities~\eqref{decoup_help1} and~\eqref{e:limit_R} imply that there exists $R_\varepsilon ' < R_\varepsilon$ such that
\begin{equation}
\label{e:first-estimate-ell}
\ell(u)  + \theta_\varepsilon\,\left[\mathcal{T}_{\Psi}(y)(u) + \varepsilon\,\mathcal{H}(\ell)(u)\right] \leq R_{\varepsilon} \qquad \text{ whenever $\ell(u)\in [R_\varepsilon ',R_\varepsilon]$.}
\end{equation}
If $\ell(u)\leq R_\varepsilon '$, by~$({\rm T}3)$ and by~\eqref{e:2-bounds-H} we estimate
\begin{equation}
\label{e:above_R}
\begin{split}
\ell(u) + \theta_\varepsilon  \cR^\varepsilon_{\Psi}(y)(u) &\leq R_\varepsilon ' + \theta_\varepsilon\,\left[C_\cT\,\omega(R_\varepsilon) +\varepsilon R_\varepsilon  k_\varepsilon + \frac{\varepsilon}{e}\right]\,.
\end{split}
\end{equation}
It follows from~\eqref{e:above_R} that there exists $\theta_\varepsilon^1 \in (0, +\infty)$ such that for every $\theta_{\varepsilon} \in (0, \theta_{\varepsilon}^{1}]$
\begin{equation}
\label{e:second-estimate-ell}
\ell(u)  + \theta_\varepsilon\,\left[\mathcal{T}_{\Psi}(y)(u) + \varepsilon\,\mathcal{H}(\ell)(u)\right] \leq R_{\varepsilon} \qquad \text{whenever $\ell(u) \in (r_{\varepsilon}, R'_{\varepsilon}]$}.
\end{equation}
Combining~\eqref{e:first-estimate-ell} and~\eqref{e:second-estimate-ell} we deduce the upper bound~\eqref{e:to-show-1} for $\theta_{\varepsilon} \in (0, \theta_{\varepsilon}^{1}]$.

We now show that, for a suitable choice of $\theta_{\varepsilon} \in (0, \theta^{1}_\varepsilon]$, we can as well guarantee 
\begin{equation}
\label{e:to-show-2}
\ell(u) + \theta_\varepsilon \cR_{\Psi}^\varepsilon(y)(u)\geq r_\varepsilon \qquad \text{$\eta$-a.e.~$u\in U $}\,.
\end{equation}
In fact, using $({\rm T}3)$ and~\eqref{entropy_lower_bound}
\begin{equation*}
\begin{split}
\ell(u) + \theta_\varepsilon \left[\cT_{\Psi}(y)(u) + \varepsilon \mathcal{H}(\ell)(u)\right] &\geq \ell(u) + \theta_\varepsilon\,\left[-C_\cT\,\omega(\ell(u))-\varepsilon\,\ell(u)\log(\ell(u))\right]
\end{split}.
\end{equation*}
If $\ell(u) \in \big( \frac{4}{3}\,r_\varepsilon, R_{\varepsilon} \big]$, by monotonicity of~$\omega$ we continue in the previous inequality with
\begin{equation}
\label{e:l-lower-r}
\begin{split}
\ell(u) + \theta_\varepsilon \left[\cT_{\Psi}(y)(u) + \varepsilon \mathcal{H}(\ell)(u)\right] &\geq \frac{4}{3}\,r_\varepsilon + \theta_\varepsilon\,\left[-C_\cT\,\omega(R_\varepsilon) - \varepsilon R_\varepsilon \log(R_\varepsilon)\right]\,.
\end{split}
\end{equation}
From inequality~\eqref{e:l-lower-r} we infer the existence of~$\theta_\varepsilon^2 \in (0, \theta_{\varepsilon}^{1} ]$ (depending only on~$r_{\varepsilon}$ and~$R_{\varepsilon}$) such that for every~$\theta_{\varepsilon} \in (0, \theta_{\varepsilon}^{2} ]$ it holds
\begin{equation}
\label{e:l-lower-r-2}
\ell(u) + \theta_\varepsilon \left[\cT_{\Psi}(y)(u) + \varepsilon \mathcal{H}(\ell)(u)\right] \geq r_{\varepsilon} \qquad \text{whenever $\ell(u) \in \Big(\frac{4}{3}r_{\varepsilon}, R_{\varepsilon} \Big]$.}
\end{equation}
If $\ell(u) \in \big[r_{\varepsilon},  \frac{4}{3}\,r_\varepsilon \big]$, instead, by~$({\rm T}3)$ and by the choice of~$r_{\varepsilon}$ in~\eqref{r_eps}, we estimate
\begin{equation}
\label{e:l-lower-r-3}
\begin{split}
\ell(u) + \theta_\varepsilon \left[\cT_{\Psi}(y)(u) + \varepsilon \mathcal{H}(\ell)(u)\right] &\geq \ell(u) + \theta_\varepsilon \ell(u)\left[-C_\cT\,\frac{\omega(\ell(u))}{\ell(u)} - \varepsilon \log(\ell(u))\right]
\\ 
&
\geq \ell(u) + \theta_\varepsilon \ell(u)\left[- C_\cT \frac{\omega\left(\frac{4}{3}\,r_\varepsilon\right)}{r_\varepsilon} + \varepsilon \log\left(\frac{3}{4\,r_\varepsilon}\right) \right] 
\\ 
&
\geq \ell(u) \geq r_{\varepsilon}\,,
\end{split}
\end{equation}
which concludes the proof of~\eqref{e:to-show-2} for $\theta_{\varepsilon} \in (0, \theta_{\varepsilon}^{2}]$. 

Combining~\eqref{e:mean=1}, \eqref{e:to-show-1}, and~\eqref{e:to-show-2}, we conclude that for every $\theta_{\varepsilon} \in (0, \theta_{\varepsilon}^{2}]$, for every $\Psi \in \mathcal{P}(B^{Y_{\varepsilon}}_{\varrho})$, and every $y = (x, \ell) \in B^{Y_{\varepsilon}}_{\varrho}$,~\eqref{e:T4-R} holds.
Notice, in particular, that~$\theta_\varepsilon$ is independent of~$\varrho$.
\end{proof}

From now on, whenever a choice of $r_{\varepsilon}$ and $R_{\varepsilon}$ is made according to Proposition~\ref{regularity_properties_entropic_vector_field}, the corresponding space $Y_{r_{\varepsilon},R_{\varepsilon}}$ will be denoted by $Y_{\varepsilon}$. Moreover, for any $N\in\mathbb{N}$, we will denote by $Y_{\varepsilon}^N\coloneqq (Y_{\varepsilon})^N$ the cartesian product of $N$ copies of $Y_{\varepsilon}$. Finally, we will consistently use the notation $b_{\Psi}^{\varepsilon}$ for the velocity field introduced in \eqref{entropic_vector_field}.

As a consequence of Theorem~\ref{Brezis_cor} and Proposition~\ref{regularity_properties_entropic_vector_field}, we obtain the following theorem.



\begin{theorem}
\label{Entropic_Problem_theorem}
Let $v_\Psi\colon Y \to \mathbb{R}^d$ satisfy $\mathrm{(v1)}$--$\mathrm{(v3)}$ and let $\cT_\Psi\colon Y \to L^p(U,\eta)$ satisfy $\mathrm{(T1)}$--$\mathrm{(T3)}$; let $\varepsilon>0$ and let $r_{\varepsilon},R_{\varepsilon}$ be as in Proposition~\ref{regularity_properties_entropic_vector_field}. 
Then for any choice of initial conditions $\bar\by = (\bar{y}^{1}, \ldots, \bar{y}^{N}) \in Y_\varepsilon^{N}$, the system
\begin{equation}\label{coup_entropic_problem_compact}
\begin{cases}
\dot{y}^{i}_t = b_{\Lambda_t^{N}}^\varepsilon(y^{i}_t), \\
y^{i}_0=\bar y^{i},
\end{cases}
\qquad \text{for $i=1,\dots,N$, $t\in [0,T]$,}
\end{equation}
where $\Lambda_t^{N}\coloneqq \frac{1}{N}\sum_{i=1}^N \delta_{y^{i}_{t}}$ is the empirical measure associated with the system, has a unique solution $\by\colon[0,T]\to Y_{\varepsilon}^{N}$.  
Moreover, 
we have that
\begin{equation}\label{coup_solution_bound}
\sup_{\substack{i=1,\ldots,N\\t\in[0,T]}}\lVert y_t^i \rVert_{\overline{Y}} \leq \Big(\sup_{i=1,\ldots,N} \lVert \bar{y}^i \rVert_{\overline{Y}} + M_\varepsilon T\Big) e^{2M_\varepsilon T}.
\end{equation}
\end{theorem}

\begin{proof}
We let $\by \coloneqq (y^{1},\dots, y^{N})\in Y_\varepsilon^N\subset \overline{Y}^N$, whose norm we define as
\begin{equation*}
\norm{\by}_{\overline{Y}^N}\coloneqq \frac{1}{N}\sum_{i=1}^N \vert\vert y^{i}\vert\vert_{\overline{Y}},
\end{equation*}
and we consider the associated empirical measure $\Lambda^{N} \coloneqq \frac{1}{N}\sum_{i=1}^N \delta_{y^{i}}$\,, which belongs to $\cP(B_R^{Y_\varepsilon})$ whenever $\boldsymbol{y}\in (B_R^{Y_\varepsilon})^N$. Consider the map $\bb^{\varepsilon,N} \colon Y_\varepsilon^N \to \overline{Y}^N$ whose components are defined through $b_i^{\varepsilon,N}(\by) \coloneqq b^\varepsilon_{\Lambda^N}(y^{i})$.
Then the Cauchy problem \eqref{coup_entropic_problem_compact} can be written as 
\begin{equation*}
\begin{cases}
\dot{\by}_t = \bb^{\varepsilon,N}(\by_t),\\ 
\by_0 =\bar{\by}.
\end{cases}
\end{equation*}

In order to apply Theorem~\ref{Brezis_cor} to the system above, we first notice that assumption (ii) is automatically satisfied since the system is autonomous. 
To see that the other assumptions are satisfied too, we fix a ball $B_R^{Y_\varepsilon^N}$ and notice that $B_R^{Y_\varepsilon^N}\subset \big(B_{NR}^{Y_\varepsilon}\big)^N$. 
Applying \eqref{coup_help1} with $\Psi = \Lambda^{N}$ to each component $y^{i}$ of $\by$, we get that assumption (iii) of Theorem~\ref{Brezis_cor} is satisfied with $\varrho = RN$. 
We now show that assumption (i) holds. 
Fix $\by_1,\by_2\in B_R^{Y_\varepsilon^N}$ and let $\Lambda^{N}_1$ and $\Lambda^{N}_2$ be the associated empirical measures. 
Recalling \eqref{W1}, we notice that 
\begin{equation*}
\mathcal{W}_1(\Lambda_1^{N},\Lambda_2^{N}) \leq \frac{1}{N} \sum\limits_{i=1}^N \lVert y^{i}_1-y^{i}_2\rVert_{\overline{Y}} = \lVert \by_1-\by_2 \rVert_{\overline{Y}^N}.
\end{equation*}
Therefore, by triangle inequality, \eqref{coup_help2}, and \eqref{coup_help3}, we obtain the estimate
\begin{equation*}
\begin{split}
\lVert \bb^{\varepsilon,N}(\by_1)- \bb^{\varepsilon, N}(\bb_2)\rVert_{\overline{Y}^N} &= \frac{1}{N}\sum_{i=1}^N \lVert b^\varepsilon_{\Lambda^N_1}(y^{i}_1)-b^\varepsilon_{\Lambda^N_2}(y^{i}_2) \rVert_{\overline{Y}} \\
& \leq  L_{NR}\,\mathcal{W}_1(\Lambda_1^{N},\Lambda_2^{N})+\frac{L_{\varepsilon,NR}}{N}\sum_{i=1}^N\lVert y_1^{i}-y^{i}_2 \rVert_{\overline{Y}} \\
& \leq (L_{NR}+L_{\varepsilon,NR}) \lVert \by_1-\by_2 \rVert_{\overline{Y}^N}\,.
\end{split}
\end{equation*}
To see that also assumption (iv) of Theorem~\ref{Brezis_cor} holds, we apply \eqref{coup_help4}, upon noticing that $m_1(\Lambda^{N})=\norm{\by}_{\overline{Y}^N}$,
\begin{equation*}
\lVert \boldsymbol{b}^{\varepsilon,N}(\by) \rVert_{\overline{Y}^N} = \frac{1}{N}\sum_{i=1}^N \lVert b^\varepsilon_{\Lambda^{N}}(y^{i}) \rVert_{\overline{Y}} \leq \frac{M_\varepsilon}{N}\sum_{i=1}^N(1+\lVert y^{i} \rVert_{\overline{Y}}+m_1(\Lambda^{N})) = M_\varepsilon\,(1+2\lVert \by \rVert_{\overline{Y}^N}).
\end{equation*}
Existence and uniqueness of the solution to system \eqref{coup_entropic_problem_compact} follow now from Theorem~\ref{Brezis_cor}.

Finally, because of \eqref{coup_help4}, 
we have that 
\begin{equation*}
\begin{split}
\lVert y_t^i \rVert_{\overline{Y}} &\leq  \lVert \bar{y}^i \rVert_{\overline{Y}}+\int_0^T \lVert b^\varepsilon_{\Lambda_s^{N}}(y^{i}_s) \rVert_{\overline{Y}}\,\de s \leq \lVert \bar{y}^i \rVert_{\overline{Y}} + \int_0^T \big[M_\varepsilon(1+\lVert y^{i}_s \rVert_{\overline{Y}}+m_1(\Lambda_s^{N}))\big]\,\de s 
\\ 
&\leq \lVert \bar{y}^i \rVert_{\overline{Y}} + \int_0^T \big[M_\varepsilon(1+\lVert y^{i}_s \lVert_{\overline{Y}} + \lVert \by_s \rVert_{\overline{Y}^N})\big]\,\de s 
\\
& \leq \sup_{j=1,\ldots,N} \lVert \bar{y}^j \rVert_{\overline{Y}} +\int_0^T \Big[M_\varepsilon\Big(1+2\sup_{j=1,\ldots,N}\lVert y_s^j \rVert_{\overline{Y}}\Big)  \Big]\,\de s .
\end{split}
\end{equation*}
Taking the supremum over $i=1,\ldots,N$ in the left-hand side and applying Gr\"{o}nwall's Lemma, we conclude that 
\begin{equation*}
\sup_{\substack{i=1,\ldots,N\\t\in[0,T]}}\lVert y_t^i \rVert_{\overline{Y}} \leq \Big(\sup_{i=1,\ldots,N} \lVert \bar{y}^i \rVert_{\overline{Y}} + M_\varepsilon T\Big) e^{2M_\varepsilon T},
\end{equation*}
which is \eqref{coup_solution_bound}.
\end{proof}

We state here a second existence and uniqueness result, which will be useful in the next section.

\begin{proposition}
\label{lagrangian_ODE_prop}
Let $v_\Psi\colon Y \to \mathbb{R}^d$ satisfy $\mathrm{(v1)}$--$\mathrm{(v3)}$ and let $\cT_\Psi\colon Y \to L^p(U,\eta)$ satisfy $\mathrm{(T1)}$--$\mathrm{(T3)}$; let $\varepsilon>0$ and let $r_{\varepsilon},R_{\varepsilon}$ be as in Proposition~\ref{regularity_properties_entropic_vector_field}. 
Let $\Lambda \in C^0([0,T]; (\cP_1(Y_\varepsilon), \mathcal{W}_1))$ and assume that there exists $\varrho>0$ such that $\Lambda_{t}\in\cP(B_{\varrho}^{Y_\varepsilon})$ for all $t\in [0,T]$. Then, for every $\bar{y}\in Y_\varepsilon$ the Cauchy problem
\begin{equation}
\label{lagrangian_ODE} 
\left\{
\begin{array}{ll}
\dot{y}_{t} = b^{\varepsilon}_{\Lambda_{t}} (y_{t})\,,\\
y_{0} = \bar{y}
\end{array}
\right.
\end{equation}
has a unique solution.
\end{proposition}

\begin{proof}
The result follows by a direct application of Theorems~\ref{Brezis_cor} and~\ref{regularity_properties_entropic_vector_field}, as this time the field~$b_{\Lambda_{t}}^{\varepsilon}$ is fixed. 
\end{proof}

In view of the previous result, the following definition is justified. 

\begin{definition}
\label{transition_map}
    Let $\varepsilon>0$, let $r_{\varepsilon},R_{\varepsilon}$ be as in Proposition~\ref{regularity_properties_entropic_vector_field}, let $\varrho>0$, and let $\Lambda \in C([0, T]; (\cP_{1} (Y_{\varepsilon}); \mathcal{W}_{1}))$ be such that $\Lambda_{t} \in \cP( B^{Y_{\varepsilon}}_{\varrho})$ for every $t \in [0, T]$. We define the transition map $\boldsymbol{\mathrm{Y}}_{\Lambda}(t,s,\bar{y})$ associated with the ODE~\eqref{lagrangian_ODE} as
\begin{equation}
\boldsymbol{\mathrm{Y}}_{\Lambda}(t, s, \bar{y}) := y_t \,,
\end{equation}
where $t\mapsto y_t $ is the unique solution to~\eqref{lagrangian_ODE}  where we have replaced the initial condition by $y_s = \bar{y}$.
\end{definition}

\section{Mean-field limit}\label{ELS}

In this section we aim at passing to the mean-field limit as $N\rightarrow \infty$ in system~\eqref{coup_entropic_problem_compact}. Along the whole section, we fix~$\varepsilon>0$, $r_{\varepsilon} \in (0, 1)$, and $R_{\varepsilon} \in (1, +\infty)$ as in Theorem~\ref{Entropic_Problem_theorem}. As it is customary in the study mean-field limits of particles systems, we look at the limit of the empirical measure $\Lambda_t^{N} = \frac{1}{N} \sum_{i=1}^{N} \delta_{y^{i}_{t}}$ associated to a solution $\by\colon [0, T] \to Y_{\varepsilon}^{N}$ of system~\eqref{coup_entropic_problem_compact}. In Theorem~\ref{Entropic_Main} we will show that, under suitable assumptions on the initial conditions, the sequence of curves $t \mapsto \Lambda^{N}_{t}$ converges to a curve $\Lambda \in C( [0, T] ; (\mathcal{P}_{1}(Y_{\varepsilon}) ; \mathcal{W}_{1}))$ solution to the continuity equation
\begin{equation}\label{continuity_equation}
\partial_t \Lambda_t + \mathrm{div}(b^\varepsilon_{\Lambda_t}\,\Lambda_t) = 0\,.
\end{equation}
We start by recalling the definition of Eulerian solution to~\eqref{continuity_equation}.

\begin{definition}
\label{d:eulerian}
Let $\bar{\Lambda} \in \cP_1 (Y_\varepsilon)$. We say that $\Lambda \in C^0([0,T];(\cP_1(Y_\varepsilon),\mathcal{W}_1))$ is an Eulerian solution to equation~\eqref{continuity_equation} with initial datum~$\bar{\Lambda}$ if~$\Lambda_{0} = \bar{\Lambda}$ and for every $\phi \in C^1_b([0,T]\times \overline{Y})$ it holds
\begin{equation}\label{continuity_equation_dual}
\int_{Y_\varepsilon}\phi(t,y)\,\de\Lambda_t (y) - \int_{Y_\varepsilon}\phi(0,y)\,\de\Lambda_0 (y) = \int_0^t \int_{Y_\varepsilon}(\partial_t\phi(s,y) + D\phi(s,y)\cdot b^\varepsilon_{\Lambda}(y))\,\de\Lambda_s (y)\,\de s,
\end{equation}
where $D\phi(s,y)$ is the Fr\'echet differential of~$\phi$ in the $y$-variable.
\end{definition}

The main result of this section is an existence and uniqueness result of Eulerian solutions to~\eqref{continuity_equation} and its characterization as the mean-field limit of the particles system~\eqref{coup_entropic_problem_compact}.

\begin{theorem}
\label{Entropic_Main}
Let $\varrho>0$ and $\bar{\Lambda}\in \cP(B_{\varrho}^{Y_\varepsilon})$ be a given initial datum. Then, the following facts hold:
\begin{enumerate}
\item there exists a unique Eulerian solution $\Lambda \in C([0, T]; (\cP_{1} (Y_{\varepsilon}); \mathcal{W}_{1}))$ to \eqref{continuity_equation} with initial datum $\bar{\Lambda}$;
\item if $\bar{\by}_{N} := (\bar{y}^{1}_{N}, \ldots, \bar{y}^{N}_{N}) \in Y^{N}_{\varepsilon}$ satisfies $\| \bar{y}^{i}_{N}\|_{\overline{Y}} \leq \varrho$ for every $i = 1, \ldots, N$ and every $N \in \mathbb{N}$ and $\bar{\Lambda}^{N} := \frac{1}{N}\sum_{i=1}^N\delta_{\bar{y}_{i,N}} \in \cP(B_r^{Y_\varepsilon})$ is such that
\begin{equation*}
\lim_{N\rightarrow \infty} \mathcal{W}_1(\bar{\Lambda} , \bar{\Lambda}^{N}) = 0\,,
\end{equation*}
then the corresponding sequence of empirical measures~$\Lambda_{t}^{N}$ associated to the system~\eqref{coup_entropic_problem_compact} with initial data $\bar{y}^{i}_{N}$ fulfill
\begin{equation*}
\lim_{N\rightarrow\infty} \mathcal{W}_1 (\Lambda_t,\Lambda_t^{N}) = 0 \qquad\textrm{uniformily with respect to}\,\,t\in[0,T].
\end{equation*}
\end{enumerate}
\end{theorem}

Before proving existence of an Eulerian solution, we briefly discuss its uniqueness. 
This result is a consequence of the following superposition principle (see~\cite[Theorem~3.11]{MS2020} and~\cite[Theorem~5.2]{Ambrosio}).

\begin{theorem}[Superposition principle]
\label{Superposition} 
Let $(E,\norm{\cdot}_{E})$ be a separable Banach space, let $b\colon (0,T)\times E \to E$ be a Borel vector field, and let $\mu \in C ([0, T];  \cP(E)) $ be such that 
\begin{equation}
\label{e:hp-supprin}
\int_0^T \int_E \norm{b_t}_E\,\de\mu_t\,\de t < +\infty \,.
\end{equation}
If $\mu$ is a solution to the continuity equation
\begin{equation*}
\partial_{t} \mu_t + \mathrm{div}(b_t\,\mu_t) = 0
\end{equation*}
in duality with cylindrical functions $\phi\in C^1_b(E)$, then there exists $\boldsymbol{\eta}\in \cP(C([0,T];E))$ concentrated on absolutely continuous solutions to the Cauchy problems
$$
\left\{\begin{array}{ll}
\dot{\gamma} = b_t(\gamma)\,,\\
\gamma_{0} \in {\rm spt} \mu_{0}
\end{array} \right.
$$
and with $(\mathrm{ev}_t)_\#\boldsymbol{\eta}=\mu_t$ for all $t\in [0,T]$, where $\mathrm{ev}_{t} \colon C([0, T]; E) \to E$ is the evaluation map at time~$t$, defined as $\mathrm{ev}_{t} (\gamma):= \gamma(t)$ for every $\gamma \in C([0, T]; E)$.
\end{theorem}

The following uniqueness result holds.

\begin{theorem}
\label{thm:uniqueness}
Let $\bar{\Lambda} \in \cP_{1} (Y_{\varepsilon})$ and assume that $\Lambda \in C([0, T]; ( \cP(Y_{\varepsilon}); \mathcal{W}_{1}))$ is a solution to~\eqref{continuity_equation} with initial condition~$\Lambda_{0} = \bar{\Lambda}$. Then, $\Lambda$ is the unique solution to~\eqref{continuity_equation} with the same initial value.
\end{theorem}

\begin{proof}
Uniqueness of~$\Lambda$ follows from Theorems~\ref{Superposition} and~\ref{Entropic_Problem_theorem}. Indeed, we notice that by continuity of~$t \mapsto \Lambda_{t}$ there exists finite
\begin{displaymath}
M:= \max_{t \in [0, T]} \, m_{1}(\Lambda_{t}) <+\infty\,.
\end{displaymath}
Hence, setting~$b_{t}: = b_{\Lambda_{t}}$ we have by~\eqref{coup_help4} that
\begin{align*}
\int_{0}^{T} \int_{Y} \| b_{t} (y) \|_{L^{p}(U, \eta)} \, \de \Lambda_{t} (y) \, \de t & \leq \int_{0}^{T} \int_{Y} M_{\varepsilon} (1 + \| y \|_{\overline{Y}} + M) \, \de \Lambda_{t} (y) \, \de t 
\\
&
\leq M_{\varepsilon} + 2 M M_{\varepsilon} < +\infty\,,
\end{align*}
which is precisely~\eqref{e:hp-supprin}. Since $L^{p} (U, \eta)$ is a separable Banach space, we may apply Theorem~\ref{Superposition} and deduce that there exists $\boldsymbol{\eta} \in \mathcal{P} (C([0, T]; \overline{Y})$ concentrated on solutions to the Cauchy problem
\begin{equation}
\label{e:Cauchy-super}
\left\{
\begin{array}{ll}
\dot{y}_{t} = b^{\varepsilon}_{\Lambda_{t}} (y_{t})\,,\\[2mm]
y_{0} \in {\rm spt} (\bar{\Lambda})\,,
\end{array}
\right.
\end{equation}
and such that $\Lambda_{t} = ({\rm ev}_{t})_{\#} \boldsymbol{\eta}$ for $t \in [0, T]$. As $\bar{\Lambda} \in \cP_{1} (Y_{\varepsilon})$, Theorem~\ref{Entropic_Problem_theorem} implies that for any initial condition~$y_{0} \in {\rm spt} (\bar{\Lambda})$ system~\eqref{e:Cauchy-super} admits a unique solution. This yields the uniqueness of~$\Lambda$.
\end{proof}

In order to prove existence of a Eulerian solution $\Lambda$ to~\eqref{continuity_equation}, we need to pass through the notion of Lagrangian solution, which we recall below (see also~\cite[Definition~3.3]{CCR2011}).

\begin{definition}
\label{d:lagrangian}
Let $\bar{\Lambda} \in\cP_{1}(Y_\varepsilon)$ be a given initial datum. We say that $\Lambda \in C^0( [0,T] ; (\cP_{1} (Y_\varepsilon) ; \mathcal{W}_1))$ is a {\em Lagrangian solution} to \eqref{continuity_equation}  with initial datum $\bar{\Lambda}$ if it satisfies 
\begin{equation}
\Lambda_t = \boldsymbol{\mathrm{Y}}_{\Lambda}(t,0,\cdot)_\#\bar{\Lambda} \qquad\textrm{for every $t \in [0, T]$},
\end{equation}
where $\boldsymbol{\mathrm{Y}}_{\Lambda}(t,s,\bar{y})$ are the transition maps associated with the ODE~\eqref{lagrangian_ODE}.
\end{definition}

\begin{remark}
\label{coup_help18}
Recalling the definition of push-forward measure, it can be directly proven that Lagrangian solutions are also Eulerian solutions. 
\end{remark}

We first need the following lemma.
\begin{lemma}
\label{l:4.7}
Let $v_\Psi\colon Y \rightarrow \mathbb{R}^d$ satisfy $\mathrm{(v1)}$--$\mathrm{(v3)}$ and let $\cT_\Psi\colon Y \rightarrow L^{p}(U, \eta)$ satisfy $\mathrm{(T1)}$--$\mathrm{(T3)}$. 
Let~$\delta>0$, let $\bar{\Lambda} \in \cP(B_\delta^{Y_\varepsilon})$, 
and assume that $\Lambda \in C^0([0,T];(\cP_1(Y_\varepsilon),\mathcal{W}_1))$ is a Lagrangian solution to~ \eqref{continuity_equation} with initial datum $\bar{\Lambda}$. Then, there exists~$\varrho \in (0, +\infty)$ only depending on~$\varepsilon$,~$\delta$, and~$T$ such that
\begin{equation*}
\Lambda_{t}\in \cP(B_{\varrho}^{Y_\varepsilon})\qquad\textrm{for  every $t\in [0,T]$.}
\end{equation*}
\end{lemma}

\begin{proof}
It suffices to show that there exists~$\varrho \in (0, +\infty)$ such that
\begin{equation}
\label{e:Gronwall-Y}
\max\limits_{y\in B_{\delta}^{Y_\varepsilon}}\norm{\boldsymbol{\mathrm{Y}}_{\Lambda}(t,0,y)}_{\overline{Y}} \leq \varrho \qquad \text{for every $t\in [0,T]$.}
\end{equation}
We first observe that by definition of Lagrangian solutions and the fact that $\bar{\Lambda} \in \cP (B_{\delta}^{Y_\varepsilon})$, we immediately have 
\begin{equation}
\label{coup_help11}
m_1(\Lambda_t)\leq \max_{y\in B_{\delta}^{Y_\varepsilon}}\norm{\boldsymbol{\mathrm{Y}}_{\Lambda} (t,0,y)}_{\overline{Y}} \qquad \text{for every $t\in [0,T]$.}
\end{equation}
Arguing as in Theorem~\ref{Entropic_Problem_theorem}, by definition of the transition map, by~\eqref{coup_help4}, and by~\eqref{coup_help11}, for every $y \in B^{Y_{\varepsilon}}_{\delta}$ we have that
\begin{equation*}
\begin{split}
\norm{\boldsymbol{\mathrm{Y}}_{\Lambda}(t,0,y)}_{\overline{Y}}&\leq \delta + M_\varepsilon \int_0^T(1+\norm{\boldsymbol{\mathrm{Y}}_{\Lambda^\varepsilon}(s,0,y)}_{\overline{Y}}+m_1(\Lambda_s^\varepsilon))\,\de s
\\ 
&
\leq  \delta + M_\varepsilon \int_0^T \Big( 1 + 2 \max_{y \in B^{Y_{\varepsilon}}_{\delta}} \norm{\boldsymbol{\mathrm{Y}}_{\Lambda}(s,0,y)}_{\overline{Y}} \Big)\,\de s\,.
\end{split}
\end{equation*}
By Gr\"{o}nwall inequality we deduce that~\eqref{e:Gronwall-Y} holds true with $\varrho =(\delta + M_{\varepsilon} T ) e^{2M_{\varepsilon}T}$. 
\end{proof}

We are now in a position to prove Theorem~\ref{Entropic_Main}.

\begin{proof}[Proof of Theorem \ref{Entropic_Main}]
The structure of the proof follows step by step that of~\cite[Theorem~3.5]{MS2020} (see also~\cite[Theorem~4.1]{Ambrosio}). We report it here briefly for the reader convenience, underlying the use of different function spaces. In particular, we notice that closed and bounded subsets of $L^{p}(U, \eta)$ are not compact, which does not allow us to apply Ascoli-Arzel\`a Theorem in combination to Theorem~\ref{Entropic_Problem_theorem} to obtain a mean-field limit result.

The proof goes through a finite-dimensional approximation and involves three steps.

\textit{Step 1: Stability of Lagrangian solutions.} 
Let us fix $\delta > 0$ and $\bar{\Lambda}^{1},  \bar{\Lambda}^{2} \in \cP (B_{\delta}^{Y_\varepsilon})$. Let us assume that $\Lambda^{1}, \Lambda^{2} \in C([0, T]; (\cP_{1}(Y_{\varepsilon}, \mathcal{W}_{1} ) )$ are two Lagrangian solutions to~\eqref{continuity_equation} with initial data $\bar{\Lambda}^{1}$ and~$\bar{\Lambda}^{2}$, respectively. In particular, by Lemma~\ref{l:4.7} we have that there exists~$\varrho$ (only depending on~$\delta$ and~$\varepsilon$) such that $\Lambda^{1}_{t}, \Lambda^{2}_{t} \in \cP(B^{Y_{\varepsilon}}_{\varrho})$ for every $t \in [0, T]$. 
We claim that 
\begin{equation}
\label{coup_help14}
\mathcal{W}_1(\Lambda_t^{1},\Lambda_t^{2}) \leq e^{L_{\varepsilon,\varrho} t +  L_{\varrho} t e^{L_{\varepsilon,\varrho}T} }  \,\mathcal{W}_1(\bar{\Lambda}^{1},\bar{\Lambda}^{2})\qquad\textrm{for every $t\in[0,T]$.}
\end{equation}
To prove~\eqref{coup_help14}, we fix $\bar{y}^{1}, \bar{y}^{2}\in B_{\delta}^{Y_\varepsilon}$ and first observe that by Lemma~\ref{l:4.7}
\begin{equation}\label{coup_help12}
\max_{t \in [0, T]} \, \vert\vert\boldsymbol{\mathrm{Y}}_{\Lambda^{i}}(t,0,\bar{y}^{\,i})\vert\vert_{\overline{Y}} \leq \varrho \qquad \text{for  $i=1,2$.}
\end{equation}
For simplicity, let us set $y^{i}_{t} := \boldsymbol{Y}_{\Lambda^{i}} (t, 0 , \bar{y}^{i})$. By \eqref{coup_help2} and \eqref{coup_help3} of Proposition~\ref{regularity_properties_entropic_vector_field} and by~\eqref{coup_help12}, we get that for every $t \in [0, T]$
\begin{equation}
\label{e:y1y2}
\begin{split}
\| y^1_t - y^2_t\|_{\overline{Y}} &\leq \| \bar{y}^{1} - \bar{y}^{2} \|_{\overline{Y}} + \int_0^t\left( \| b^\varepsilon_{\Lambda_s^{1}}(y^1_s) -  b^\varepsilon_{\Lambda_s^{1}}(y^2_s) \|_{\overline{Y}} + \| b^\varepsilon_{\Lambda_s^{1}}(y^2_s) - b^\varepsilon_{\Lambda_s^{2}}(y^2_s)\|_{\overline{Y}}\right)\de s 
\\ 
&
\leq \| \bar{y}^{1} - \bar{y}^{2}\|_{\overline{Y}} + L_{\varrho} \int_0^t \mathcal{W}_1(\Lambda_s^{1} , \Lambda_s^{2})\, \de s + \int_0^t L_{\varepsilon,\varrho}\,\| y^1_s - y^2_s \|_{\overline{Y}}\,\de s\,.
\end{split}
\end{equation}
Applying Gr\"{o}nwall's lemma, we infer from~\eqref{e:y1y2} that for every $t \in [0, T]$
\begin{equation}
\label{coup_help13}
\| y^1_t - y^2_t \|_{\overline{Y}} \leq \left(  \| \bar{y}^{1} - \bar{y}^{2} \|_{\overline{Y}} + L_{\varrho} \int_0^t  \mathcal{W}_1(\Lambda_s^{1}  ,\Lambda_s^{2}) \, \de s \right)\,e^{L_{\varepsilon, \varrho} t} \,.
\end{equation}

Let~$\Pi \in \mathcal{P} (\overline{Y} \times \overline{Y})$ be an optimal plan between $\bar{\Lambda}^{1}$ and $\bar{\Lambda}^{2}$. By the definition of Lagrangian solutions, $(\boldsymbol{\mathrm{Y}}_{\Lambda^{1}}(t,0,\cdot),\boldsymbol{\mathrm{Y}}_{\Lambda^{2}}(t,0,\cdot))_\#\Pi$ is a transport plan between $\Lambda_t^{1}$ and $\Lambda_t^{2}$. Therefore, using~\eqref{coup_help13} we may estimate
\begin{equation*}
\begin{split}
\mathcal{W}_1(\Lambda^{1}_t,\Lambda^{2}_t) & \leq \int_{Y_\varepsilon\times Y_\varepsilon} \| \boldsymbol{\mathrm{Y}}_{\Lambda^{1}}(t,0,y^1) -\boldsymbol{\mathrm{Y}}_{\Lambda^{2}}(t,0,y^2) \|_{\overline{Y}}\,\de\Pi(y^1,y^2) 
\\
&
\leq  e^{L_{\varepsilon, \varrho} t } \int_{\overline{Y} \times \overline{Y} } \|  y^1 - y^2 \|_{\overline{Y}} \, \de \Pi (y^1,y^2) + L_{\varrho} e^{L_{\varepsilon,\varrho} t } \int_0^t \mathcal{W}_1(\Lambda_s^{1} , \Lambda_s^{2} ) \, \de s 
\\ 
&
= e^{L_{\varepsilon,\varrho} t }  \mathcal{W}_1(\bar{\Lambda}^{1},\bar{\Lambda}^{2}) + L_{\varrho} e^{L_{\varepsilon,\varrho}t}\,\int_0^t \mathcal{W}_1(\Lambda_s^{1},\Lambda_s^{2}) \,\de s\,.
\end{split}
\end{equation*}
Applying again the Gr\"{o}nwall lemma we deduce~\eqref{coup_help14}.

\textit{Step 2: Existence and approximation of Lagrangian solutions.} We fix a sequence of atomic measures $\bar{\Lambda}^{N } \in \cP(B_{\delta}^{Y_{\varepsilon}} )$ such that
\begin{equation}
\label{coup_help16}
\lim_{N\to  \infty}\mathcal{W}_1(\bar{\Lambda}^{N},\bar{\Lambda} ) = 0 \,.
\end{equation} 
Such a sequence can be constructed as follows: let $\bar{y}^{i}(z)\in Y_\varepsilon$ be independent and identically distributed with law $\bar{\Lambda}$, so that the random measures $\bar{\Lambda}^{N} \coloneqq \frac{1}{N}\sum_{i=1}^N \delta_{\bar{y}^i(z)}$ almost surely converge in~$\cP_1(Y_\varepsilon)$ to $\bar{\Lambda}$. Then, choose a realization~$z$ such that this convergence takes place. By Theorem~\ref{Entropic_Problem_theorem}, there exists unique the solution to system~\eqref{coup_entropic_problem_compact} with initial condition $\bar{\by} = (\bar{y}^{1}, \ldots, \bar{y}^{N})$ and let $\Lambda^{N}_t$ be the associated empirical measures. As $\Lambda_t^{N}$ are also Lagrangian solutions to~\eqref{continuity_equation} with initial condition $\bar{\Lambda}^{N}$,~\eqref{coup_help14} provides a constant $C\coloneqq C(\varepsilon, \delta, T)$ such that for every $t\in [0,T]$ and every $N, M \in \mathbb{N}$
\begin{equation*}
\mathcal{W}_1(\Lambda_t^{N},\Lambda_t^{M}) \leq C \mathcal{W}_1(\bar{\Lambda}^{N}  , \bar{\Lambda}^{M} )\,.
\end{equation*}
Thus, $\Lambda^{N}\in C([0,T];(\cP_1(B_{\varrho}^{Y_\varepsilon}),\mathcal{W}_1))$ is a Cauchy sequence, and there exists $\Lambda \in C([0,T];(\cP_1(B_{\varrho}^{Y_\varepsilon}),\mathcal{W}_1))$ such that $\Lambda^{N}_{t}$ converges to~$\Lambda_{t}$ with respect to the Wasserstein distance~$\mathcal{W}_{1}$, uniformly in $t \in [0, T]$. Moreover, arguing as in the proof of~\eqref{e:Gronwall-Y}, we may find~$\bar{\varrho} \geq \varrho$ such that $\boldsymbol{Y}_{\Lambda} (t, 0, \bar{y}) \in B^{Y_{\varepsilon}}_{\bar{\varrho}}$ for every $t \in [0, T]$ and every $\bar{y} \in B^{Y_{\varepsilon}}_{\delta}$. 
In view of~\eqref{coup_help2} and~\eqref{coup_help3} we obtain that 
\begin{equation*}
\| \boldsymbol{\mathrm{Y}}_{\Lambda } ( t , 0 , \bar{y} ) - \boldsymbol{\mathrm{Y}}_{\Lambda^{N}}(t,0,\bar{y} )\|_{\overline{Y}} \leq  L_R\,e^{L_{\varepsilon,\bar{\varrho} }t} \int_0^t  \mathcal{W}_1(\Lambda_s ,\Lambda_{s}^{N}) \,\de s\,,
\end{equation*}

\textit{Step 3: Uniqueness and conclusion.} Uniqueness of Lagrangian solutions, given the initial datum, follows now from \eqref{coup_help14}. Uniqueness of Eulerian solutions is stated in Theorem~\ref{thm:uniqueness}.
\end{proof}

\section{Fast Reaction Limit for undisclosed replicator-type dynamics}
\label{Fast_Reaction_Limit}

The aim of this section is to address the case in which the dynamics for the labels runs at a much faster time scale than the dynamics for the agents' positions. In this case, introducing the fast time scale $\tau = \lambda\,t$, with $\lambda \gg 1$, system  \eqref{coup_entropic_problem_compact} takes the form
\begin{equation}
\label{two_time_scale_system}
\begin{cases}
\dot{x}_t^{i} = v_{\Lambda_t^N}(x_t^{i},\ell^{i}_t), \\
\dot{\ell}_t^{i} = \lambda[\cT_{\Lambda_t^{N}}(x^{i}_t,\ell^{i}_t) + \varepsilon\,\mathcal{H}(\ell_t^{i})]
\end{cases} \qquad \textrm{for }i=1,\dots,N,\,\,t\in[0,T].
\end{equation}
Note that, for $\varepsilon>0$ and $0 < r_{\varepsilon} < 1 < R_{\varepsilon} < +\infty$ as in Proposition~\ref{regularity_properties_entropic_vector_field}, the well-posedness of~\eqref{two_time_scale_system} is still guaranteed by Theorem~\ref{Entropic_Problem_theorem} (see Proposition~\ref{p:well-posedness-lambda}). We focus on the behavior of system \eqref{two_time_scale_system} as $\lambda\rightarrow +\infty$, thus we are interested in the case of instantaneous adjustment of the strategies.

From now on, for $\Psi \in \cP_{1} (Y_{\varepsilon})$ we denote $\nu \coloneqq \pi_{\#} \Psi$, where $\pi \colon Y_{\varepsilon} \to \R^{d}$ is the canonical projection over~$\R^{d}$. If $\Lambda^{N}, \Lambda$ are curves with values in~$\cP_{1}(Y_{\varepsilon})$, the symbols~$\mu^{N}$ and~$\mu$ will instead indicate the curves of measures $\mu^{N}_{t}, \mu_{t}$, obtained as push-forward of~$\Lambda^{N}_{t}$ and~$\Lambda_{t}$ for $t \in [0, T]$ through~$\pi$. 

We assume that the strategies dynamics is of replicator type, i.e., we suppose that in the second equation in~\eqref{two_time_scale_system} the operator~$\mathcal{T}_{\Psi}$ takes the form
\begin{equation}
\label{fast_reaction}
\mathcal{T}_{\Psi} (x, \ell) \coloneqq \left( \int_{U} \partial_\xi F_{\nu} ( x , \ell (u) , u ) \ell (u) \, \de \eta (u) - \partial_\xi F_{\nu } ( x , \ell , \cdot ) \right ) \ell \quad \text{for $x \in \mathbb{R}^{d}$ and $\ell \in L^{p}(U, \eta)$,}
\end{equation}
for a map $F\colon \mathcal{P}_1(\mathbb{R}^d)\times \mathbb{R}^d \times ( 0 , +\infty) \times U \to [-\infty,+\infty]$
satisfying the following properties:
\begin{itemize}
\item[$(\mathrm{F1})$] for every $\varrho>0$, every $\nu \in \mathcal{P}(B_\varrho)$, every $x\in B_\varrho$, and every $\ell\in C_\varepsilon$, the map $u \mapsto F_{\nu} ( x , \ell(u) , u )$ is $\eta$--integrable;

\item[$(\mathrm{F2})$]  for every $\varrho>0$, every $\nu \in \mathcal{P}(B_\varrho)$, every $x\in B_\varrho$, and every~$u\in U$, the map $g_{(\nu, x, u)} \colon (0, +\infty) \to \mathbb{R}$ defined as $g_{(\nu, x, u)}(\xi) \coloneqq  F_{\nu} ( x , \xi , u )$ is convex, is differentiable, and its derivative~$g'_{(\nu, x, u)}$ is Lipschitz continuous in~$(0, +\infty)$, uniformly with respect of~$(\nu, x, u) \in  \mathcal{P}(B_\varrho) \times B_{\varrho}\times U$; 

\item[$(\mathrm{F3})$]  there exists $C_{F}>0$ such that for every $\varrho>0$, every $\nu \in \mathcal{P}(B_{\varrho})$, every $x \in B_{\varrho}$, every $\xi \in (0, +\infty)$, and every $u \in U$
\begin{displaymath}
|\partial_{\xi} F_{\nu} (x, \xi, u) | \leq C_{F}\,;
\end{displaymath}


\item[$(\mathrm{F4})$] for every $\varrho>0$, the maps $(\nu, x) \mapsto F_{\nu} (x, \xi, u)$ and $(\nu, x) \mapsto \partial_{\xi} F_{\nu} (x, \xi, u)$ are Lipschitz continuous in~$\mathcal{P}_{1}(B_{\varrho}) \times B_{\varrho}$ uniformly with respect to~$u \in U$ and~$\xi \in (0, +\infty)$. Namely, there exists $\Gamma_{\varrho}>0$ such that for every $\xi \in (0, +\infty)$, every $x_{1}, x_{2} \in B_{\varrho}$, every $\nu_{1}, \nu_{2} \in \cP(B_{\varrho})$, and every $u \in U$
\begin{align*}
| F_{\nu_{1}}(x_1,\xi,u) - F_{\nu_{2}}(x_2,\xi,u) | & \leq \Gamma_{\varrho} \big( |  x_1 - x_2 | + \mathcal{W}_1(\nu_{1} ,\nu_{2} ) \big)\,,\\
 | \partial_\xi F_{\nu_{1}} (x_1, \xi , u ) - \partial_\xi F_{\nu_2} ( x_2 , \xi , u ) | & \leq \Gamma_{\varrho}  \big ( | x_1 - x_2 | + \mathcal{W}_1(\nu_1,\nu_2) \big)\,;
\end{align*}

\item[$(\mathrm{F5})$] for every $\varrho>0$, every $\nu \in \mathcal{P}(B_{\varrho})$, every $\xi \in (0, +\infty)$, and every $u \in U$, the map $F_{\nu} (\cdot, \xi, u)$ is differentiable in~$\mathbb{R}^{d}$.
\end{itemize}

\begin{remark}\label{rem5.1}
The analysis of the fast reaction limit in the undisclosed setting has been recently performed in~\cite{entropy} for the replicator dynamics (see also Remark~\ref{r:examples}), where the authors considered a pay-off function~$J$ independent of the strategy~$u'$ played by other players. Hence, the functional~$\mathcal{J}_{\Psi}$ in~\eqref{e:JLambda} takes the form
\begin{equation*}
\begin{split}
\mathcal{J}_{\Psi} (x,u) = \int_{Y} J(x,u,x')\,\de\Psi (x',\ell') = \int_{\mathbb{R}^d} J(x,u,x')\,\de \nu(x') \eqqcolon \mathcal{J}_{\nu}(x,u)\,,
\end{split}
\end{equation*}
which would correspond (see~\eqref{fast_reaction}) to the operator
\begin{equation*}
\cT_\Psi (x,\ell) = \left(\mathcal{J}_{\nu} (x,\cdot) - \int_U \mathcal{J}_{\nu} (x,u) \ell(u) \,\de\eta(u) \right)\ell\,.
\end{equation*}
and to $F_{\nu} (x, \xi, u) \coloneqq -\mathcal{J}_{\nu}(x,u) \xi$ for every $(\nu, x, \xi, u) \in \mathcal{P}_{1}(\mathbb{R}^{d}) \times \mathbb{R}^{d} \times (0, +\infty) \times U$. Furthermore, in~\cite{entropy} a precise choice for the velocity field~$v_{\Psi}$ is made, which is independent of the state variable~$\Psi$. 

The theoretical framework described in~$\mathrm{(F1)}$--$\mathrm{(F5)}$ is more flexible than~\cite{entropy}. Besides the freedom in the choice of~$v_{\Psi}$, we may for instance model more involved situations, where the pay-off of a certain strategy depends as well on how often such strategy has been played. Such behavior may be captured by a pay-off function $\widetilde{J} \colon \mathbb{R}^{d} \times U \times \mathbb{R}^{d} \times (0, +\infty) \to \mathbb{R}$ of the form
\begin{displaymath}
\widetilde{J} (x, u, x', \xi) \coloneqq J(x, u, x') - J_{1} (\xi)\,,
\end{displaymath}
where $J_{1} \colon [0, +\infty) \to \mathbb{R}$ is monotone increasing, concave, and differentiable with bounded and Lipschitz derivative. In particular, the monotonicity assumption of~$J_{1}$ is meant to penalize strategies that are played too often, and may be therefore expected by other players. Monotonicity of~$J_{1}$ and the regularity of its derivatives comply with conditions~$\mathrm{(F1)}$--$\mathrm{(F5)}$.
\end{remark}

The following proposition provides a set of conditions under which assumptions $(\mathrm{F1})$--$(\mathrm{F5})$ are satisfied for integral functionals.

\begin{proposition}
\label{integralcase}
Let $f\colon \mathbb{R}^d\times  ( 0 , +\infty)  \times U \times \mathbb{R}^d \to (-\infty,+\infty]$ satisfy the following properties:
\begin{itemize}
\item[$\mathrm{(f1)}$] for every $\varrho>0$, every $\nu\in \cP(B_{\varrho})$, every $x\in B_{\varrho}$, and every $\ell\in L^{p}(U, \eta)$ the map
\begin{equation*}
u\mapsto \int_{\mathbb{R}^d} f(x,\ell(u),u,x')\,\de \nu (x')
\end{equation*}
is $\eta$--integrable;

\item[$\mathrm{(f2)}$]  for every $\varrho>0$,  every $x, x' \in B_\varrho$, and every~$u\in U$, the map $\xi \mapsto f(x, \xi, u, x')$ is convex in~$(0, +\infty)$, is differentiable with derivative $\partial_{\xi} f(x, \xi, u, x')$ Lipschitz continuous in~$(0, +\infty)$, uniformly with respect to~$(x, u, x') \in B_{\varrho} \times U \times B_{\varrho}$;


\item [$(\mathrm{f3})$]  there exists $C_{f}>0$ such that for every $\varrho>0$, $x, x' \in B_{\varrho}$, every $\xi \in (0, +\infty)$, and every $u \in U$
\begin{displaymath}
|\partial_{\xi} f (x, \xi, u, x') | \leq C_{f}\,.
\end{displaymath}

\item[$\mathrm{(f4)}$] for every $\varrho>0$, every $x, x' \in B_{\varrho}$, every $ \xi \in (0, +\infty)$, and every~$u\in U$ the function $x' \mapsto f(x,\xi,u,x')$ belongs to $\mathrm{Lip}_b(\mathbb{R}^d)$ and the map $x \mapsto f (x , \xi , u , x' )\in \mathrm{Lip}(\mathbb{R}^d)$, with Lipschitz constants dependent only on~$\varrho$;

\item[$\mathrm{(f5)}$] for every $\varrho>0$, every $x, x' \in B_{\varrho}$, every $ \xi \in (0, +\infty)$, and every~$u\in U$, the function $x' \mapsto \partial_\xi f(x,\xi,u, x')$ belongs to~$\mathrm{Lip}_b(\mathbb{R}^d)$, and the map $x \mapsto \partial_\xi f ( x , \xi , u ,x' )$ belongs to $\mathrm{Lip}(\mathbb{R}^d)$, with Lipschitz constants depending only on~$\varrho$;

\item[$\mathrm{(f6)}$] for every $\varrho>0$, every $\xi \in (0, +\infty)$, every $u \in U$, and every $x' \in B_{\varrho}$, the map $f(\cdot, \xi, u, x')$ is differentiable in~$\mathbb{R}^{d}$.
\end{itemize}
Then, the functional $F \colon \mathcal{P}_1(\mathbb{R}^d)\times \mathbb{R}^d \times ( 0 , +\infty) \times U \to (-\infty,+\infty]$ defined as
\begin{equation*}
F_\nu (x,\xi,u) \coloneqq \int_{\mathbb{R}^d} f ( x , \xi , u , x' )\,\de \nu(x')
\end{equation*}
fulfills conditions~$\mathrm{(F1)}$--$\mathrm{(F5)}$.
\end{proposition}

\begin{proof}
Condition $\mathrm{(F1)}$ coincides with $\mathrm{(f1)}$. Property~$\mathrm{(F2)}$ follows from $\mathrm{(f2)}$, which in particular implies that
\begin{displaymath}
\partial_{\xi} F_{\nu} (x, \xi, u) = \int_{\R^{d} } \partial_{\xi} f (x, \xi, u, x') \, \de \nu(x')\,.
\end{displaymath}
Thus, we deduce $(\mathrm{F3})$ and~$(\mathrm{F4})$ from~$\mathrm{(f3)}$--$\mathrm{(f5)}$. Finally, from $\mathrm{(f5)}$ and~$\mathrm{(f6)}$ we deduce that for every~$\varrho>0$, every $\xi \in (0, +\infty)$, every $u \in U$, and every $\nu \in \mathcal{P}(B_{\varrho})$ we have
\begin{displaymath}
\partial_{x} F_{\nu} (x, \xi, u) = \int_{U} \partial_{x} f(x, \xi, u, x') \, \de \nu(x')\,. \qedhere
\end{displaymath}
\end{proof}

For $\lambda\in (0, +\infty)$, we now briefly discuss the well-posedness of~\eqref{two_time_scale_system} for the operator~$\mathcal{T}_{\Psi}$ as in~\eqref{fast_reaction}

\begin{proposition}
\label{p:well-posedness-lambda}
Let~$F \colon \mathcal{P}_1(\mathbb{R}^d)\times \mathbb{R}^d \times ( 0 , +\infty) \times U \to (-\infty,+\infty]$ satisfy $(\mathrm{F1})$--$(\mathrm{F5})$. Then, the operator~$\mathcal{T}_{\Psi}$ defined in~\eqref{fast_reaction} for every $\Psi \in \mathcal{P}_{1}(Y)$ satisfies conditions~$\mathrm{(T1)}$--$\mathrm{(T3)}$.
\end{proposition}

\begin{proof}
By definition~\eqref{fast_reaction},~$\mathcal{T}_{\Psi}$ clearly satisfies~$(\mathrm{T1})$. Property~$(\mathrm{T2})$ is a consequence of~$(\mathrm{F2})$ and of~$(\mathrm{F4})$, while~$(\mathrm{T3})$ follows from~$(\mathrm{F3})$, as for $y = (x, \ell) \in Y$ and $u \in U$ we can simply estimate
\begin{displaymath}
| \mathcal{T}_{\Psi} (y) (u)| \leq 2 C_{F} | \ell(u)|\,.
\end{displaymath}
Thus, $(\mathrm{T3})$ is satisfied with~$\omega( \xi ) \coloneqq | \xi|$ for $\xi \in [0, +\infty)$.
\end{proof}

\begin{corollary}
\label{c:lambda}
Let $v_{\Psi}$ satisfy $\mathrm{(v1)}$--$\mathrm{(v3)}$, let~$F \colon \mathcal{P}_1(\mathbb{R}^d)\times \mathbb{R}^d \times ( 0 , +\infty) \times U \to (-\infty,+\infty]$ satisfy $(\mathrm{F1})$--$(\mathrm{F5})$. and let~$\mathcal{T}_{\Psi}$ be as in~\eqref{fast_reaction}. Moreover, for~$\varepsilon>0$ let $0< r_{\varepsilon} < 1 < R_{\varepsilon} < +\infty$ be given by Proposition~\ref{regularity_properties_entropic_vector_field}. Then, the following facts hold:
\begin{itemize}
\item[$(i)$] for every $\lambda \in (0, +\infty)$ and every $N \in \mathbb{N}$, system~\eqref{two_time_scale_system} admits a unique solution for every initial condition~$\bar{y} \coloneqq (\bar{y}^{1}, \ldots, \bar{y}^{N}) \in Y_{\varepsilon}^{N}$;

\item[$(ii)$] for every $\lambda, \delta \in (0, +\infty)$ and every $\bar{\Lambda} \in \mathcal{P}(B^{Y_{\varepsilon}}_{\delta})$, there exists a unique (Lagrangian / Eulerian) solution to the continuity equation
\begin{equation}
\label{e:continuity_lambda}
\partial_{t} \Lambda_{t} + \dive (b^{\varepsilon, \lambda}_{\Lambda_{t}} \Lambda_{t}) = 0 \qquad \text{with $\Lambda_{0}= \bar{\Lambda}$,}
\end{equation}
where we have set
\begin{displaymath}
b^{\varepsilon, \lambda}_{\Lambda_{t}} (y) \coloneqq \left( 
\begin{array}{cc}
v_{\Lambda_{t}} (y) \\
\lambda ( \mathcal{T}_{\Lambda_{t}} (y) + \varepsilon \mathcal{H} (\ell)) 
\end{array}\right);
\end{displaymath}

\item[$(iii)$] for every $\lambda, \delta \in (0, +\infty)$ and every $\bar{\Lambda}, \bar{\Lambda}_{n} \in \mathcal{P}(B^{Y_{\varepsilon}}_{\delta})$ such that $\mathcal{W}_{1}( \bar{\Lambda}_{n} , \bar{\Lambda}) \to 0$ as $n \to \infty$, the corresponding solutions~$\Lambda, \Lambda_{n} \in C([0, T] ; (\mathcal{P}_{1}(Y^{\varepsilon}), \mathcal{W}_{1}))$ to~\eqref{e:continuity_lambda} with initial conditions~$\bar{\Lambda}$ and~$\bar{\Lambda}_{n}$, respectively, satisfy
\begin{displaymath}
\lim_{n \to \infty} \, \mathcal{W}_{1} (\Lambda_{n, t} , \Lambda_{t}) = 0 \qquad \text{uniformly in $t \in [0, T]$.}
\end{displaymath}
\end{itemize}
\end{corollary}

\begin{proof}
All the items are a consequence of Proposition~\ref{p:well-posedness-lambda} and of Theorem~\ref{Brezis_cor}, and can be obtained arguing as in  Proposition~\ref{regularity_properties_entropic_vector_field} and Theorems~\ref{Entropic_Problem_theorem} and~\ref{Entropic_Main}, taking care of the fact that all the involved constants ($L_{\varrho}$, $L_{\varepsilon, \varrho}$, $M_{\varepsilon}$, and~$\theta_{\varepsilon}$) may depend on~$\lambda$.
\end{proof}

As we did in Section~\ref{sec_decoupled}, from now on we fix~$\varepsilon>0$ and $0 < r_{\varepsilon} < 1 < R_{\varepsilon} < +\infty $ as in Proposition~\ref{regularity_properties_entropic_vector_field} (or, equivalently, as in Proposition~\ref{p:well-posedness-lambda}). We recall that we set $C_{\varepsilon} \coloneqq C_{r_{\varepsilon}, R_{\varepsilon}}$ and $Y_{\varepsilon} \coloneqq Y_{r_{\varepsilon}, R_{\varepsilon}}$.

Our goal is to prove the convergence, as $\lambda \rightarrow +\infty$, of system \eqref{two_time_scale_system} to a suitable system of agents with labels, where such labels are defined as minima of some particular functionals. In Proposition \ref{G} we introduce the prototype for these functionals and present some of its properties. Before stating Proposition \ref{G}, we recall the definition of Fr\'{e}chet differentiability on~$C_{\varepsilon}$ (see, e.g.,~\cite[Appendix~A.1]{Ambrosio}).

\begin{definition}[Fr\'{e}chet differentiability]
\label{Frechet-differentiability}
Let us set $E_{C_{\varepsilon}} \coloneqq \mathbb{R} (C_{\varepsilon} - C_{\varepsilon})$. A functional $\mathcal{F}\colon C_{\varepsilon} \to \mathbb{R}$ is said to be Fr\'echet differentiable at $\ell\in C_{\varepsilon}$ if there exists $L \in \mathcal{L}(E_{C_{\varepsilon}}; \mathbb{R})$ such that 
\begin{equation*}
    \lim_{\substack{\tilde{\ell} \,\xrightarrow{L^p}\,\ell\\ \tilde{\ell} \in C_{\varepsilon}}} 
    \frac{\vert \mathcal{F}(\tilde{\ell})-\mathcal{F}(\ell)- L [\tilde{\ell}-\ell] \vert}{ \| \tilde{\ell}-\ell \|_{L^p(U,\eta)}}=0\,.
\end{equation*}
\end{definition}

\begin{remark}
Notice that the linear operator~$L$ in Definition~\ref{Frechet-differentiability} is not uniquely determined on~$E_{C_{\varepsilon}}$, while it is unique on the cone~$E_{\ell} \coloneqq \mathbb{R}_{+} ( C_{\varepsilon} - \ell)$. For this reason, we will always use the notation $D \mathcal{F} (\ell)$ to denote the operator~$L$.
\end{remark}

\begin{proposition}
\label{G}
Let $F \colon \mathcal{P}_1(\mathbb{R}^d)\times \mathbb{R}^d \times ( 0 , +\infty) \times U \to (-\infty,+\infty]$ satisfy~$(\mathrm{F1})$--$(\mathrm{F5})$. For every $\varrho > 0$, every $\nu \in \mathcal{P}(B_{\varrho} )$, and every $x\in B_{\varrho}$,  let $G_{\nu} (x , \cdot ) \colon C_\varepsilon \to \mathbb{R}$ be defined by
\begin{equation}\label{Gfun}
G_{\nu} (x,\ell) \coloneqq  \int_{U} \big( F_{\nu}(x,\ell(u),u) + \varepsilon \ell(u) ( \log (\ell(u))  - 1) \big)\,\de \eta(u) \qquad \text{for $\ell \in C_{\varepsilon}$.} 
\end{equation}
Then, $G_{\nu} (x, \cdot)$ is Fr\'echet differentiable if $p \geq 1$, strongly convex if $1\leq p \leq 2$ and uniformly convex if $2<p<+\infty$. Moreover, there exists $D_{\varrho}>0$ such that for every $\ell_{1}, \ell_{2}\in C_\varepsilon$ and every $(x_{1}, \nu_{1}), (x_{2}, \nu_{2}) \in B_{\varrho}\times \mathcal{P}(B_{\varrho})$\begin{equation}
\label{e:better_Lip}
|G_{\nu_{1}} (x_{1}, \ell_{1} ) - G_{\nu_{2}} (x_{2}, \ell_{2}) | \leq D_{\varrho} \big( | x_{1} - x_{2}| + \| \ell_{1} - \ell_{2} \| _{L^{p}(U, \eta)} + \mathcal{W}_{1} (\nu_{1}, \nu_{2}) \big)\,.
\end{equation}
\end{proposition}

\begin{proof}
For $(x, \nu) \in B_{\varrho} \times  \mathcal{P} (B_{\varrho})$, the functional $G_{\nu} (x,\cdot)$ is well-defined thanks to~$(\mathrm{F1})$. Furthermore,  as a consequence of~$(\mathrm{F2})$, $G_{\nu} (x,\cdot)$ is Fr\'echet-differentiable in $\ell_1\in C_\varepsilon$ with differential
\begin{equation*}
\begin{split}
DG_\nu (x,\ell_1)[\ell_{1} - \ell_{2}] = \int_{U} \big( \partial_\xi F_{\nu} ( x , \ell_1(u) , u ) (\ell_2(u) - \ell_{1}(u)) + \varepsilon ( (\ell_{2} (u) - \ell_{1}(u)) \log (\ell_{1}(u)) \big)  \,\de\eta(u) \,.
\end{split}
\end{equation*}

Indeed, by~$(\mathrm{F2})$ we can simply estimate
\begin{equation*}
    \begin{split}
        | &G_\nu (x,\ell_2) - G_\nu (x,\ell_1) - DG_\nu (x,\ell_1) [\ell_2 - \ell_1] | 
        \\ 
        & = \bigg| \int_U [F_{\nu} ( x , \ell_2(u) ,u ) - F_{\nu} ( x , \ell_1(u) , u ) - \partial_\xi F_{\nu} ( x , \ell_1(u) , u ) ( \ell_2 (u) - \ell_1 (u) ) ] 
        \\
        &
        \qquad + \varepsilon \ell_{2}(u) ( \log (\ell_{2}(u))  - 1) - \varepsilon \ell_{1}(u) ( \log (\ell_{1}(u))  - 1) - \varepsilon ( (\ell_{2} (u) - \ell_{1}(u)) \log (\ell_{1}(u)) \big) \, \de \eta(u) \bigg|
        \\
        &
        \leq o(1) \int_U | \ell_{1}(u) - \ell_{2}(u)| \, \de \eta(u) \leq o(1) \|  \ell_1 - \ell_2 \|_{L^p(U,\eta)}\,.
    \end{split}
\end{equation*}

By the local strong convexity of $t \mapsto \log t$ in $(0, +\infty)$, there exists $\beta_{\varepsilon}>0$ such that for every $\xi_{1}, \xi_{2} \in [r_{\varepsilon}, R_{\varepsilon}]$ and every $t \in [0, 1]$
\begin{displaymath}
(t \xi_{1} + (1-t ) \xi_{2}) \log ( t \xi_{1} + (1-t ) \xi_{2}) \leq t \xi_{1} \log \xi_{1} + (1-t) \xi_{2} \log \xi_{2} - \frac{\beta_{\varepsilon} }{2} t (1-t) | \xi_{1} - \xi_{2}|^{2}\,.
\end{displaymath}
By convexity of $F_{\nu} (x, \cdot, u)$ we deduce that for every $t \in [0, 1]$ and every $\ell_{1}, \ell_{2} \in C_{\varepsilon}$
\begin{equation}
\label{e:convexity_G}
\begin{split}
G_\nu (x, t \ell_1 + (1-t) \ell_2 ) \leq t G_{\nu} (x, \ell_{1}) + (1-t) G_{\nu} (x, \ell_{2}) - \frac{\beta_{\varepsilon}}{2} t (1-t) \| \ell_{1} - \ell_{2} \|^{2}_{L^{2} (U, \eta)}\,.
\end{split}
\end{equation}
If $p \in [1, 2]$, inequality~\eqref{e:convexity_G} implies the strong convexity of~$G_{\nu} (x, \cdot)$ in~$C_{\varepsilon}$ by H\"older inequality. If $p \in (2, +\infty)$, instead, we infer the uniform convexity of~$G_{\nu} (x, \cdot)$ by combining~\eqref{e:convexity_G} with
\begin{equation}
\label{lpnorm}
\| \ell_1 - \ell_2\|_{L^p(U,\eta)}^p \leq (R_\varepsilon-r_\varepsilon)^{p-2} \| \ell_1 - \ell_2\|_{L^2(U,\eta)}^2\,.
\end{equation}

 Finally, the Lipschitz continuity~\eqref{e:better_Lip} is a direct consequence of property $\mathrm{(F3)}$,~$\mathrm{(F4)}$, and of the local Lipschitz continuity of~$t \mapsto t \log t$ in $(0, +\infty)$.
\end{proof}

As a consequence of Proposition \ref{G} we have the following corollary. 
\begin{corollary}
\label{Corollary}
Let $F \colon \mathcal{P}_1(\mathbb{R}^d)\times \mathbb{R}^d \times ( 0 , +\infty) \times U \to (-\infty,+\infty]$ satisfy~$(\mathrm{F1})$--$(\mathrm{F5})$ and let~$G$ be defined as in~\eqref{Gfun}. Then, for every $\varrho>0$, every $\nu \in \mathcal{P}(B_{\varrho}^{Y_\varepsilon})$, every $x\in B_{\varrho}$, and every $1\leq p<+\infty$, there exists a unique solution $\ell_{x, \nu}$ to the minimum problem
\begin{equation}
\label{e:minG}
\min_{\ell\in C_\varepsilon} \,G_{\nu} (x , \ell)\,.
\end{equation}
Moreover, there exists~$\beta_{\varepsilon} >0$ and $A_{\varepsilon, \varrho} >0$ such that for every $x, x_{1}, x_{2} \in B_{\varrho}$, every $\nu, \nu_{1}, \nu_{2} \in \cP(B_{\varrho} )$, and every $\ell \in C_\varepsilon$
\begin{align}
\label{minimality}
& G_{\nu} (x,\ell) - G_\nu (x,\ell_{x, \nu})  \geq \beta_{\varepsilon} \| \ell - \ell_{x, \nu} \|^2_{L^2(U,\eta)}\,,
\\
& 
| G_{\nu_{1}}(x_{1},\ell_{x_{1}, \nu_{1}}) - G_{\nu_{2}}(x_{2} , \ell_{x_{2}, \nu_{2}}) |  \leq D_{ \varrho} \big( | x_{1} - x_{2}| + \mathcal{W}_{1} (\nu_{1}, \nu_{2})\big)\,, \label{minimality_2} 
\\
& 
\| \ell_{x_{1}, \nu_{1}} - \ell_{x_{2}, \nu_{2}}\|_{L^{p}(U, \eta)}  \leq A_{\varepsilon, \varrho} \big( | x_{1} - x_{2}| + \mathcal{W}_{1} (\nu_{1}, \nu_{2})\big) \qquad \text{if $p \in [1, 2]$}\,,\label{minimality_3}
\\
& 
\| \ell_{x_{1}, \nu_{1}} - \ell_{x_{2}, \nu_{2}}\|_{L^{p}(U, \eta)}  \leq A_{\varepsilon, \varrho} \big( | x_{1} - x_{2}| + \mathcal{W}_{1} (\nu_{1}, \nu_{2})\big)^{\frac{1}{p-1}} \qquad \text{if $p \in (2, +\infty)$}\,,\label{minimality_4}
\end{align}
where $D_{\varrho}>0$ is the Lipschitz constant introduced in Proposition~\ref{G}.
\end{corollary}

\begin{proof}
The existence and uniqueness to the minimum problem is a direct consequence of the strong and uniform convexity of $G_\nu ( x,\cdot)$ and of the convexity of $C_\varepsilon$. Then, by the minimality of $\ell_{x, \nu}$ and by the local strong convexity of~$t \mapsto t \log t$, there exists $\beta_{\varepsilon}>0$ such that for every $\ell \in C_{\varepsilon}$
\begin{equation*}
\begin{split}
G_\nu (x, \ell) - G_\nu (x,\ell_{x, \nu}) & \geq \underbrace{DG_\nu(x, \ell_{x, \nu})[ \ell - \ell_{x, \nu}]}_{\geq 0} + \beta_{\varepsilon} \| \ell - \ell_{x, \nu} \|^2_{L^2(U,\eta)} \geq \beta_{\varepsilon}\| \ell - \ell_{x, \nu}\|^2_{L^2(U,\eta)}\,,
\end{split}
\end{equation*}
which proves~\eqref{minimality}.

Let us now fix $x_1, x_2\in B_{\varrho}$, $\nu_{1}, \nu_{2} \in \cP_1(B_{\varrho})$, and let $\ell_i \in C_{\varepsilon}$ be the solutions to 
\begin{displaymath}
\min_{\ell\in C_\varepsilon} G_{\nu_{i} }(x_i, \ell) \qquad \text{for $i = 1, 2$.}
\end{displaymath}
Without loss of generality, we may assume that $G_{\nu_{2}}(x_2,\ell_2) \geq G_{\nu_{1} }(x_1,\ell_1)$. Using the minimality of~$\ell_2$ and applying Proposition~\ref{G} we get that
\begin{equation*}
\begin{split}
| G_{\nu_{2}} (x_2,\ell_2) - G_{\nu_{1}}(x_1,\ell_1) | & = G_{\nu_{2}}(x_2,\ell_2 ) - G_{\nu_{2}}(x_2,\ell_1) + G_{\nu_{2}} (x_2,\ell_1) -  G_{\nu_{1}}(x_1,\ell_1) 
\\ 
&
\leq G_{\nu_{2}}(x_2,\ell_1) - G_{\nu_{1}}(x_1,\ell_1) 
\\ 
&
\leq D_\varrho ( |x_2-x_1| + W_1(\nu_{1} , \nu_{2}))\,,
\end{split}
\end{equation*}
which yields~\eqref{minimality_2}.

Since $G_\nu(x,\cdot)$ is strongly convex in~$C_{\varepsilon}$  for $p=2$, we have that there exists $\gamma_{\varepsilon}>0$ such that
\begin{equation*}
    \big(DG_{\nu_{2}}(x_2,\ell_{2}) - DG_{\nu_{2}}(x_2,\ell_1) \big) [\ell_2-\ell_1] \geq \gamma_{\varepsilon} \| \ell_2 - \ell_1 \|^2_{L^2(U,\eta)}\,.
\end{equation*}
By minimality, we have that 
\begin{equation*}
    DG_{\nu_{2}}(x_2,\ell_2) [\ell_2 - \ell_1] \leq 0 \leq DG_{\nu_{1}}(x_1,\ell_1)[\ell_2 - \ell_1]\,.
\end{equation*}
Therefore, property $\mathrm{(F4)}$ yields 
\begin{equation}
\label{e:Lipschitz-ell}
\begin{split}
     \gamma_{\varepsilon} \| \ell_2 - \ell_1 \|^2_{L^2(U,\eta)} & \leq \big( DG_{\nu_{1}}(x_1,\ell_1) - DG_{\nu_{2}} (x_2,\ell_1)\big) [\ell_2 - \ell_1] 
     \\ 
     &
     = \int_U \big( \partial_\xi F_{\nu_{1}} ( x_1 , \ell_1(u) , u) - \partial_\xi F_{\nu_{2}} ( x_2 , \ell_1 (u) , u ) \big) (\ell_2(u)  - \ell_1 (u)) \,\de \eta(u)
     \\ 
     &
     \leq \Gamma_\varrho (| x_2 - x_1 | + \mathcal{W}_1( \nu_{1}, \nu_{2}) \big) \int_U |\ell_2 (u) - \ell_1(u) | \, \de \eta(u) 
     \\ 
     &
     \leq \Gamma_\varrho \big( | x_2 - x_1 | + \mathcal{W}_1(\nu_{1},\nu_{2}) \big) \| \ell_2 - \ell_1\|_{L^2(U,\eta)}\,.
\end{split}
\end{equation}
If $p \in [1, 2]$, \eqref{e:Lipschitz-ell} and H\"older inequality imply the Lipschitz continuity of~$(x, \nu) \mapsto \ell_{x, \nu}$ in $B_{\varrho} \times \cP(B_{\varrho})$. If $p \in (2, +\infty)$, arguing as in~\eqref{lpnorm} and using once again H\"older inequality we deduce from~\eqref{e:Lipschitz-ell} that
\begin{equation*}
    \| \ell_2 - \ell_1\|^{p}_{L^p(U,\eta)} \leq \Big(\frac{\Gamma_\varrho}{\gamma_{\varepsilon}}\Big)^{2} \,(R_\varepsilon - r_\varepsilon)^{p-2} \big ( | x_2 - x_1 | + \mathcal{W}_1(\nu_{1} , \nu_{2}) \big)^{2}\,.
\end{equation*}
Setting 
$$A_{\varepsilon, \varrho} \coloneqq \max\, \bigg \{ \frac{\Gamma_{\varrho}}{\gamma_{\varepsilon}} , \bigg( \frac{\Gamma_{\varrho}}{\gamma_{\varepsilon}} \bigg)^{\frac{2}{p}} (R_{\varepsilon} - r_{\varepsilon})^{\frac{p-2}{p}} \bigg\}
$$
 we get~\eqref{minimality_3} and \eqref{minimality_4}.
\end{proof}

As intermediate step towards the main result of this section we have the following lemma, where we estimate the behavior, as $\lambda \to +\infty$, of the labels~$\ell^{i}_{t}$ in system~\eqref{two_time_scale_system}. For later use, we introduce here the map $\Delta\colon \mathbb{R}^{d} \times \mathcal{P}_{1}(\mathbb{R}^{d}) \to C_{\varepsilon}$ defined as
\begin{displaymath}
\Delta (x, \nu) \coloneqq \argmin_{\ell \in C_{\varepsilon}} \, G_{\nu} (x, \ell)\,.
\end{displaymath}
In particular, by Proposition~\ref{G} the map $\Delta$ is Lipschitz continuous on~$B_{\varrho} \times \mathcal{P}(B_{\varrho})$ for every $\varrho>0$.

\begin{lemma}
\label{lemma}
Let $v_{\Psi}$ satisfy~$\mathrm{(v1)}$--$\mathrm{(v3)}$, let $F \colon \mathcal{P}_1(\mathbb{R}^d)\times \mathbb{R}^d \times ( 0 , +\infty) \times U \to (-\infty,+\infty]$ satisfy~$(\mathrm{F1})$--$(\mathrm{F5})$, let the operator~$\mathcal{T}_{\Psi}$ be defined as in~\eqref{fast_reaction}, and let~$G$ be as in~\eqref{Gfun}.
For $ \lambda \in (0, +\infty)$, $N \in \mathbb{N}$, $\delta>0$, and $\bar{\by} = (\bar{y}^{1}, \ldots, \bar{y}^{N}) \in (B_{\delta}^{Y^{\varepsilon}})^{N}$, let~$\{y^{i}_{\lambda}\}_{i=1}^{N}$ denote the solutions to the Cauchy problem~\eqref{two_time_scale_system} with initial conditions~$\bar{y}^{i}$ and corresponding empirical measure~$\Lambda^{N}_{\lambda, t} \coloneqq \frac{1}{N} \sum_{i=1}^{N} \delta_{y^{i}_{\lambda, t}}$, let $\bar\Lambda^{N}_{0} \coloneqq \frac{1}{N} \sum_{i=1}^{N} \delta_{\bar{y}^{i}}$, and let $\mu^{N}_{\lambda, t} \coloneqq \pi_{\#} \Lambda^{N}_{\lambda, t}$ and $\bar{\mu}^{N} \coloneqq \pi_{\#} \bar{\Lambda}^{N}$. Then, the following facts hold:
\begin{itemize}
\item[$(i)$] there exists $\varrho>0$ (depending only on~$\delta$ and~$\varepsilon$) such that $\Lambda^{N}_{\lambda, t} \in \mathcal{P}(B^{Y_{\varepsilon}}_{\varrho})$ for every $t \in [0, T]$;

\item[$(ii)$] there exists two positive constants~$\omega_{\varepsilon, \delta}$ and~$\gamma_{\varepsilon}$ (independent of~$\lambda$) such that for every  $p \in [1, 2]$ and every~$t\in (0,T]$
\begin{equation}
\label{fasthelp6}
   \| \ell_{\lambda, t}^{i } - \Delta( x^{i}_{\lambda, t}, \mu^{N}_{\lambda, t}) \|_{L^{p}(U, \eta)} \leq  \omega_{\varepsilon, \delta} \bigg( \frac{1}{\sqrt{\lambda}} +  e^{-\lambda \gamma_{\varepsilon} T} \bigg) \,,
\end{equation}
while for $p \in (2, +\infty)$ it holds
\begin{equation}
\label{fasthelp6.1}
 \frac{ \| \ell_{\lambda, t}^{i } - \Delta ( x^{i}_{\lambda, t}, \mu^{N}_{\lambda, t})  \|^{p}_{L^{p}(U, \eta)}}{(R_{\varepsilon} - r_{\varepsilon})^{p-2}} \leq  2 \omega_{\varepsilon, \delta}^{2} \bigg( \frac{1}{\lambda} + e^{-2 \lambda \gamma_{\varepsilon} T} \bigg)\,
\end{equation}
\end{itemize}
\end{lemma}

\begin{proof}
The proof consists of two steps. In the first step, we obtain some useful estimates and  properties of system \eqref{two_time_scale_system}, which we then use in the second step to prove~\eqref{fasthelp6}. Along the proof, we drop the index $\lambda$, as we always argue for a fixed parameter $\lambda \in (0, +\infty)$.

\noindent \textit{Step 1.}
We first show that the player's' locations $x^{i}_{t}$ are bounded in~$\mathbb{R}^{d}$ independently of~$\lambda$, $N$, and~$t$. Indeed, using $\mathrm{(v3)}$ and recalling that $m_1(\Lambda_{t}^{N})\leq \max_{i=1, \ldots, N} \|y^{i}_{t} \|_{\overline{Y}}$ and that $ \ell^{i}_{t} \in C_{\varepsilon}$, we have that for every $i=1, \ldots, N$
\begin{equation}
\label{e:some-Gron}
\begin{split}
| x_{t}^{i} | &\leq | \bar{x}^{i} | + \int_0^T | v_{\Lambda_{s}^{N}} ( x_{s}^{i} , \ell_{s}^{i})| \, \de s 
\leq | \bar{x}^{i} | + \int_0^T  M_v ( 1 + | x_s^{i} | + \| \ell_s^{i} \|_{L^{p}(U, \eta)} + m_1(\Lambda_s^{N}) ) \, \de s 
\\ 
&
\leq  | \bar{x}^{i} | + M_v (1 + R_\varepsilon ) T + \int_0^T  M_v ( | x_s^{i} | + \max_{i=1, \ldots, N}  \| y_s^{i} \|_{\overline{Y}^N}) \, \de s 
\\ 
&
\leq \delta + M_v (1 + 2R_\varepsilon) T + \int_0^T  2M_v  \max_{j=1, \ldots, N} | x_s^{j} | \,\de s\,.
\end{split}
\end{equation}
Taking the maximum over $i=1, \ldots, N$ on the left-hand side of~\eqref{e:some-Gron}, by Gr\"onwall inequality we get 
\begin{equation}
\label{e:some-Gron2}
\max_{i=1, \ldots, N} |x_t^{i} |\leq  \Big( \delta + M_v (1 + 2R_\varepsilon) T \Big) e^{2M_v T} \eqqcolon R_{\delta, \varepsilon}\,.
\end{equation}
As a consequence of~\eqref{e:some-Gron2}, setting $\varrho \coloneqq R_{\delta, \varepsilon} + R_{\varepsilon}$ we have that $(x^{i}_{t} , \Lambda^{N}_{t}) \in B_{\varrho} \times \mathcal{P}(B^{Y_{\varepsilon}}_{\varrho})$ for every $N$, every $i$, and every $t \in [0, T]$. In particular, this proves~$(i)$. Moreover, by~$\mathrm{(v3})$ and~\eqref{e:some-Gron2} the map $t\mapsto x_t^{i}$ for every $i=1, \ldots, N$, with Lipschitz constant only depending on~$\varrho$ and on~$M_{v}$. Indeed, for every $t_{1} < t_{2} \in [0, T]$ and every $i$ we have that
\begin{equation}
\label{fasthelp8}
\begin{split}
| x_{t_2}^{i } - x_{t_1}^{i} | & \leq \int_{t_1}^{t_2} | v_{\Lambda_s^{N}}(x_s^{i} , \ell_s^{i})| \,\de s 
\leq M_v (1 + 2\varrho ) | t_2 - t_1 | \eqqcolon A_{\varrho} | t_2 - t_1 |\,.
\end{split}
\end{equation}
Therefore, also the map $t\mapsto \mu^{N}_{t}$ is Lipschitz continuous, with Lipschitz constant $A_{\varrho}$. Up to a re-definition of~$A_{\varrho}$, by~$\mathrm{(F3})$ and the properties of~$\mathcal{H}$, we may as well assume that~$\ell^{i}_{t}$ is Lipschitz continuous in~$[0, T]$, with Lipschitz constant~$A_{\varrho}$.

\noindent \textit{Step 2.}
We now proceed with the proof of~\eqref{fasthelp6}. Using the convexity of $G_{\mu_{t}^{N}}(x_t^{i},\cdot)$ and the fact that $\ell^{i}_{t}, \Delta( x^{i}_{t}, \mu^{N}_{t})  \in C_{\varepsilon}$, we have that 
\begin{equation*}
\begin{split}
&G_{\mu_{t}^{N}} ( x_t^{i} , \ell_t^{i} ) - G_{\mu_{t}^{N}} ( x_t^{i} , \Delta( x_t^{i}, \mu^{N}_{t})  ) \leq DG_{\mu_{t}^{N}} (x_t^{i} , \ell_t^{i}) [ \ell_t^{i} - \Delta( x^{i}_{t}, \mu^{N}_{t})  ]
\\
&
= \int_U \big(  \partial_\xi F_{\mu^{N}_{t} } (x_t^{i} , \ell_t^{i} (u) , u ) + \varepsilon \log ( \ell^{i}_{t} (u) ) \big) (\ell_t^{i} (u) - \Delta( x^{i}_{t}, \mu^{N}_{t}) (u) ) \, \de \eta(u) 
\\
&
= \int_U \bigg( \partial_\xi F_{\mu^{N}_{t} } ( x_t^{i} , \ell_t^{i} (u) , u ) - \int_U \partial_\xi F_{\mu_t^{N} } ( x_t^{i} , \ell_t^{i} ( u' ) , u' ) \ell_t^{i} (u') \, \de\eta(u') \bigg) (\ell_t^{i} (u) - \Delta( x^{i}_{t}, \mu^{N}_{t})  (u) ) \, \de\eta(u)
\\
&
\qquad + \int_{U}  \varepsilon \big( \log ( \ell^{i}_{t} (u) ) -  I(\ell^{i}_{t} ) \big) (\ell_t^{i} (u) -\Delta( x^{i}_{t}, \mu^{N}_{t})  (u) ) \, \de \eta(u) 
 \\
&
\leq \bigg\|\partial_\xi F_{\mu_t^{N}} ( x_t^{i} , \ell_t^{i} , \cdot ) + \varepsilon  \log ( \ell^{i}_{t} ) - \int_U \partial_{\xi} F_{\mu_t^{N}} ( x_t^{i} , \ell_t^{i} (u') , u' ) \ell_t^{i}(u') \, \de \eta(u') - \varepsilon I(\ell^{i}_{t}) \bigg\|_{L^2(U,\eta)}  \times
\\
&
\qquad \times \| \ell_t^{i} -\Delta( x^{i}_{t}, \mu^{N}_{t})  \|_{L^2(U,\eta)}\,. 
\end{split}
\end{equation*}
The above chain of inequalities, together with~\eqref{minimality}, leads us to 
\begin{equation}
\label{fasthelp3}
\begin{split}
\beta_{\varepsilon} & \big( G_{\mu_{t}^{N}} ( x_t^{i} , \ell_t^{i} ) - G_{\mu_{t}^{N}} ( x_t^{i} , \Delta( x^{i}_{t}, \mu^{N}_{t})  ) \big)
\\
&
\leq \bigg\| \partial_\xi F_{\mu_t^{N}} ( x_t^{i} , \ell_t^{i} , \cdot ) + \varepsilon  \log ( \ell^{i}_{t} ) - \int_U \partial_{\xi} F_{\mu_t^{N}} ( x_t^{i} , \ell_t^{i} (u') , u' ) \ell_t^{i}(u') \, \de \eta(u') - \varepsilon I(\ell^{i}_{t})  \bigg\| _{L^2(U,\eta)}^2\,.
\end{split}
\end{equation}

By Proposition~\ref{G}, by~\eqref{fasthelp8}, and by the bound $y^{i}_{t} =(x^{i}_{t}, \ell^{i}_{t}) \in B_{\varrho}^{Y_{\varepsilon}}$ for $i = 1, \ldots, N$, for every $t < s \in (0, T)$ we may estimate
\begin{equation}
\label{fasthelp10.1}
\begin{split}
& \big( G_{\mu_{s}^{N}} ( x_{s}^{i} , \ell_{s}^{i} ) - G_{\mu_{s}^{N}} ( x_{s}^{i} , \Delta ( x^{i}_{s}, \mu^{N}_{s}) ) \big) -  G_{\mu_{t}^{N}} ( x_{t_{1}}^{i} , \ell_{t}^{i} ) - G_{\mu_{t}^{N}} ( x_{t}^{i} , \Delta ( x^{i}_{t}, \mu^{N}_{t} ) ) \big)
\\
&
= \big( G_{\mu_{s}^{N}} ( x_{s}^{i} , \ell_{s}^{i} )-  G_{\mu_{t}^{N}} ( x_{s}^{i} , \ell_{s}^{i} ) \big) + \big(  G_{\mu_{t}^{N}} ( x_{s}^{i} , \ell_{s}^{i} ) - G_{\mu_{t}^{N}} ( x_{t}^{i} , \ell_{t}^{i} )\big)
\\
&
\qquad - \big( G_{\mu_{s}^{N}} ( x_{s}^{i} , \Delta ( x^{i}_{s}, \mu^{N}_{s}) ) - G_{\mu_{t}^{N}} ( x_{t}^{i} ,\Delta ( x^{i}_{t}, \mu^{N}_{t} ) \big)  
\\
&
\leq D_{\varrho} \mathcal{W}_{1} (\mu^{N}_{t} , \mu^{N}_{s}) +  \big(  G_{\mu_{t}^{N}} ( x_{s}^{i} , \ell_{s}^{i} ) - G_{\mu_{t}^{N}} ( x_{t}^{i} , \ell_{t}^{i} )\big)
\\
&
\qquad - \big( G_{\mu_{s}^{N}} ( x_{s}^{i} , \Delta ( x^{i}_{s}, \mu^{N}_{s}) ) - G_{\mu_{t}^{N}} ( x_{t}^{i} ,\Delta ( x^{i}_{t}, \mu^{N}_{t} ) ) \big)  
\\
&
\leq D_{\varrho} A_{\varrho} ( s - t) +  \big(  G_{\mu_{t}^{N}} ( x_{s}^{i} , \ell_{s}^{i} ) - G_{\mu_{t}^{N}} ( x_{t}^{i} , \ell_{t}^{i} )\big)
\\
&
\qquad - \big( G_{\mu_{s}^{N}} ( x_{s}^{i} , \Delta ( x^{i}_{s}, \mu^{N}_{s}) )  - G_{\mu_{t}^{N}} ( x_{t}^{i} , \Delta ( x^{i}_{t}, \mu^{N}_{t} )  ) \big)  \,.
\end{split}
\end{equation}
Since also the map $t \mapsto G_{\mu^{N}_{t}} ( x_t^{i} , \ell_t^{i} ) - G_{\mu_{t}^{N}} ( x_t^{i} , \Delta ( x^{i}_{t}, \mu^{N}_{t} )  )$ is Lipschitz continuous (see Proposition~\ref{G} and Corollary~\ref{Corollary}), and thus differentiable a.e.~in~$[0, T]$, dividing~\eqref{fasthelp10.1} by $s-t$ and passing to the limit as $s \searrow t$ we get by chain rule that for a.e.~$t \in [0, T]$ 
\begin{equation}
\label{fasthelp10}
\begin{split}
\frac{\de}{\de t}  \big( G_{\mu_{t}^{N}} ( x_t^{i} , \ell_t^{i} ) - & G_{\mu_{t}^{N}} ( x_t^{i} , \Delta ( x^{i}_{t}, \mu^{N}_{t} )  ) \big) 
\\
\leq & D_{\varrho}A_{\varrho} +  \underbrace{ \int_{U} \big(\partial_\xi F_{\mu_{t} ^{N}}(x_t^{i} , \ell_t^{i} (u) , u ) + \varepsilon \log (\ell^{i}_{t}(u)) \big)  \dot{\ell}_t^{i} (u) \, \de \eta(u)}_{\textrm{I}} 
\\
&
+ \underbrace{\partial_x G_{\mu_{t}^{N}}( x_t^{i} , \ell_t^{i} )\cdot \dot{x}_t^{i}}_{\textrm{II}} 
\underbrace{-\frac{\de }{\de t}G_{\mu_{t}^{N}} (x_t^{i} , \Delta ( x^{i}_{t}, \mu^{N}_{t} )  )}_{\textrm{III}}
\end{split}
\end{equation}

We now show that the terms ${\textrm{II}}$ and ${\textrm{III}}$ are well-defined and uniformly bounded with respect to~$\lambda \in (0, +\infty)$. Let us start with~$\mathrm{II}$. By~$\mathrm{(F4)}$,~$\mathrm{(F5)}$,~$\mathrm{(v3)}$, and by the fact that $(x^{i}_{t} , \Lambda^{N}_{t}) \in B_{\varrho} \times \mathcal{P}(B^{Y_{\varepsilon}}_{\varrho})$, we get that
\begin{equation}
\label{fasthelp11}
\begin{split}
{\textrm{II}} &= \partial_x G_{\mu_{t}^{N}}(x_t^{i} , \ell_t^{i} ) \cdot \dot{x}_t^{i} 
= \int_U \partial_x F_{\mu_{t}^{N}}(x_t^{i} , \ell_t^{i} (u), u  ) \cdot v_{\Lambda_t^{N}} ( x_t^{i} , \ell_t^{i} ) \,\de \eta(u) 
\\ 
&
\leq \int_U \big| \partial_x F_{\mu_{t}^{N}} (x_t^{i} , \ell_t^{i} (u) , u ) \big|  \, \big|  v_{\Lambda_t^{N}} ( x_t^{i} , \ell_t^{i} ) \big| \,\de \eta(u) 
\leq \int_{U} \Gamma_{\varrho}  M_v( 1 + \| y_t^{i} \|_{\overline{Y}} + m_1(\Lambda_t^{N}) ) \, \de \eta(u) 
\\
&
\leq \Gamma_{\varrho} M_v ( 1 + 2\varrho ) = \Gamma_{\varrho} A_{\varrho} \,.
\end{split}
\end{equation}
As for $\mathrm{III}$, by~\eqref{minimality_2} of Corollary~\ref{Corollary} and by~\eqref{fasthelp8}, we have that for a.e.~$t \in [0, T]$
\begin{equation}
\label{fasthelp12}
{\textrm{III}} \leq \left\vert\frac{\de }{\de t}G_{\mu_{t}^{N}}(x_t^{i} , \Delta ( x^{i}_{t}, \mu^{N}_{t} ) ) \right\vert \leq 2 D_{\varrho}A_{\varrho}\,.
\end{equation}

We now estimate~${\textrm{I}}$ from~\eqref{fasthelp10}. Using~\eqref{two_time_scale_system},~\eqref{fast_reaction}, and \eqref{fasthelp3}, and recalling that $\ell^{i}_{t} \in C_{\varepsilon}$, we obtain that 
\begin{equation}
\label{fasthelp14}
\begin{split}
{\textrm{I}} & =  \int_U \big( \partial_\xi F_{\mu_t^{N}} ( x_t^{i} , \ell_t^{i} (u), u ) + \varepsilon \log (\ell^{i}_{t} (u)) \big) \dot{\ell}_t^{i} (u) \, \de \eta(u) 
\\ 
&
= \int_U  \bigg( \partial_\xi F_{\mu_t^{N}} ( x_t^{i} , \ell_t^{i} (u), u ) - \!\! \int_U \partial_\xi F_{\mu_t^{N}} ( x_t^{i} , \ell_t^{i} ( u' ) , u' ) \ell_t^{i}(u') \, \de \eta ( u' ) 
\\
&
  \qquad\qquad\qquad + \varepsilon \big(  \log (\ell^{i}_{t} (u)) - I(\ell^{i}_{t}) \big) \bigg)\,\dot{\ell}_t^{i}(u)  \de \eta(u)
  \\
  &
  = -\lambda \int_U \bigg( \partial_\xi F_{\mu_t^{N}} ( x_t^{i} , \ell_t^{i} (u) , u) - \int_U \partial_\xi F_{\mu_t^{N}} ( x_t^{i} , \ell_t^{i}(u') , u') \ell_t^{i} (u') \, \de \eta(u')  
  \\
  &
  \qquad\qquad\qquad + \varepsilon \big(  \log (\ell^{i}_{t} (u) ) - I(\ell^{i}_{t}) \big)  \bigg)^2 \ell_t^{i}(u)\,\de\eta(u) 
  \\ 
  &
  \leq  -\lambda  r_\varepsilon \int_U \bigg( \partial_\xi F_{\mu_t^{N}} ( x_t^{i} , \ell_t^{i} (u), u ) - \int_U \partial_\xi F_{\mu_t^{N}} ( x_t^{i} , \ell_t^{i}(u'),u') \ell_t^{i}(u') \, \de\eta(u') 
    \\
  &
  \qquad\qquad\qquad + \varepsilon \big(  \log (\ell^{i}_{t} (u)) - I(\ell^{i}_{t}) \big) \bigg)^2 \, \de \eta (u) 
  \\ 
  &
  \leq - \lambda r_\varepsilon \beta_{\varepsilon} \big( G_{\mu_{t}^{N}} ( x_t^{i} , \ell_t^{i} ) - G_{\mu_{t}^{N}} ( x_t^{i} ,\Delta ( x^{i}_{t}, \mu^{N}_{t} )  ) \big)\,.
\end{split}
\end{equation}

Combining~\eqref{fasthelp10}--\eqref{fasthelp14} and setting $K_{\varrho} \coloneqq (\Gamma_{\varrho} + 3D_{\varrho})A_{\varrho}$, we deduce that for a.e.~$t \in [0, T]$
\begin{equation*}
\begin{split}
\frac{\de }{\de t} \big(G_{\mu_{t}^{N}} ( x_t^{i} ,  \ell_t^{i}) - G_{\mu_{t}^{N}} ( x_t^{i} , \Delta ( x^{i}_{t}, \mu^{N}_{t} )  ) \big)
\leq & \ - \lambda r_\varepsilon \beta_{\varepsilon} \big( G_{\mu_{t}^{N}} ( x_t^{i} , \ell_t^{i}) - G_{\mu_{t}^{N}} ( x_t^{i} ,\Delta ( x^{i}_{t}, \mu^{N}_{t} )  ) \big) + K_{\varrho} \,.
\end{split}
\end{equation*}
or equivalently
\begin{equation*}
\begin{split}
\frac{\de }{\de t }  \bigg( G_{\mu_{t}^{N}} ( x_t^{i} , \ell_t^{i} ) &  - G_{\mu_{t}^{N}} ( x_t^{i} , \Delta ( x^{i}_{t}, \mu^{N}_{t} ) ) - \frac{K_{\varrho} }{\lambda\,\beta_{\varepsilon} r_\varepsilon } \bigg) 
\\
&
\leq - \lambda r_\varepsilon \beta_{\varepsilon} \bigg( G_{\mu_{t}^{N}} ( x_t^{i} , \ell_t^{i} ) - G_{\mu_{t}^{N}} ( x_t^{i} ,\Delta ( x^{i}_{t}, \mu^{N}_{t} )  ) - \frac{K_{\varrho} }{\lambda\,\beta_{\varepsilon} r_\varepsilon } \bigg) .
\end{split}
\end{equation*}
Therefore, by Gr\"{o}nwall's lemma we deduce that for every $t \in [0, T]$
\begin{equation*}
\begin{split}
 G_{\mu_{t}^{N}} ( x_t^{i} , \ell_t^{i} ) & - G_{\mu_{t}^{N}} ( x_t^{i} ,\Delta ( x^{i}_{t}, \mu^{N}_{t} )  ) - \frac{K_{\varrho}}{\lambda\,\beta_{\varepsilon} r_\varepsilon }
\\
&
\leq \bigg( G_{\bar{\mu}^{N}} ( \bar{x}^{i} , \bar{\ell}^{i} ) - G_{\bar{\mu}^{N} } ( \bar{x}^{i} , \Delta (\bar{x}^{i}, \bar{\mu}^{N} ) ) - \frac{K_{\varrho}}{\lambda \beta_{\varepsilon} r_{\varepsilon}} \bigg) e^{-\lambda r_\varepsilon \beta_{\varepsilon}T} .
\end{split}
\end{equation*}
Using~\eqref{minimality}, \eqref{minimality_2}, and the fact that $\bar{y}^{i} \in B_{\delta}^{Y_{\varepsilon}}$ and $\bar{\mu}^{N} \in \mathcal{P}(B_{\delta})$, we further obtain
\begin{equation*}
\begin{split}
\beta_{\varepsilon} \| \ell_t^{i} -\Delta ( x^{i}_{t}, \mu^{N}_{t} )  \|_{L^2(U,\eta)}^2 & \leq \frac{K_{\varrho} }{ \lambda \beta_{\varepsilon} r_\varepsilon} 
+ \bigg(G_{\bar{\mu}^{N}} (\bar{x}^{i} , \bar{\ell}^{i} ) - G_{\bar{\mu}^{N}} ( \bar{x}^{i} , \Delta ( \bar{x}^{i}, \bar{\mu}^{N} )  ) - \frac{K_{\varrho}}{\lambda \beta_{\varepsilon} r_\varepsilon} \bigg) e^{-\lambda r_\varepsilon \beta_{\varepsilon} T}
\\
&
\leq  \frac{K_{\varrho} }{ \lambda \beta_{\varepsilon} r_\varepsilon} + \Big( D_{\delta}  \| \bar{\ell}^{i} - \Delta (\bar{x}^{i}, \bar{\mu}^{N})\|_{L^{2}(U, \eta)}  - \frac{K_{\varrho}}{\lambda \beta_{\varepsilon} r_\varepsilon} \bigg) e^{-\lambda r_\varepsilon \beta_{\varepsilon} T}
\\
&
\leq \frac{K_{\varrho} }{ \lambda \beta_{\varepsilon} r_\varepsilon} +  2 D_{\delta} R_{\varepsilon}  e^{-\lambda r_\varepsilon \beta_{\varepsilon} T}
\,.
\end{split}
\end{equation*}
Recalling that $\varrho$ only depends on~$\delta$ and~$\varepsilon$, setting
$$
\omega_{\varepsilon, \delta} \coloneqq\max\, \bigg\{ \sqrt{\frac{K_{\varrho} }{ \beta_{\varepsilon}^{2} r_\varepsilon}} ; \sqrt{\frac{2D_{\delta} R_{\varepsilon}}{\beta_{\varepsilon}}} \bigg\}   \,, \qquad \gamma_{\varepsilon} \coloneqq \frac{r_{\varepsilon} \beta_{\varepsilon}}{2}
$$
we infer~\eqref{fasthelp6} for $p = 2$, and thus for every $p \in [1, 2]$ by H\"older inequality, with $\omega_{\varepsilon, \delta}^{p}\coloneqq \omega_{\varepsilon, \delta}$ and $\gamma_{\varepsilon}^{p} \coloneqq \gamma_{\varepsilon}$. For $p \in (2, +\infty)$ we recall that
$$
\| \ell^{i}_{t} - \Delta (x^{i}_{t}, \mu^{N}_{t} ) \|^{p}_{L^{p}(U, \eta)} \leq (R_{\varepsilon} - r_{\varepsilon})^{p-2} \| \ell_t^{i} - \Delta( x^{i}_{t}, \mu^{N}_{t}) \|_{L^2(U,\eta)}^2\,.
$$
which implies~\eqref{fasthelp6.1}.
\end{proof}


To simplify the notation, we define $w_{\nu}(x) \coloneqq v_{(id, \Delta)_{\#} \nu} (x, \Delta(x, \nu))$ for $x \in \mathbb{R}^{d}$ and $\nu \in \mathcal{P}_{1} (\mathbb{R}^{d})$. We now discuss the convergence of solutions to~\eqref{two_time_scale_system} to solutions to the \emph{fast reaction} system
\begin{equation}
\label{limit_sym}
\left\{\begin{array}{lll}
\dot{{x}}_t^{i} = w_{\mu^{N}_{t}} ({x}_t^{i} )\,,\\ 
{x}_0^{i}=\bar{x}^{i}
\end{array} \right. \qquad \textrm{for }i=1,\dots,N,\,\,t\in(0,T]\,,
\end{equation}
where we have set $\mu^{N}_{t} \coloneqq \frac{1}{N} \sum_{i=1}^{N} \delta_{x^{i}_{t}}$. 
We start with the basic properties of~$w_{\mu}$ and the well-posedness of~\eqref{limit_sym}.

\begin{lemma}
\label{l:W}
The following facts hold:
\begin{itemize}
\item[$(i)$] for every $\varrho>0$ there exists~$\Xi_{\varrho}>0$ such that for every $\nu_{1}, \nu_{2} \in \mathcal{P}( B_{\varrho})$ and every $x_{1}, x_{2} \in B_{\varrho}$
\begin{align*}
| w_{\nu_{1}} (x_{1}) - w_{\nu_{2}} (x_{2}) | & \leq \Xi_{\varrho} \big( |x_{1}- x_{2}| + \mathcal{W}_{1}(\nu_{1}, \nu_{2})\big)\,;
\end{align*}

\item[$(ii)$] there exists $M_{w}>0$ such that  the velocity field~$w_{\mu} (x)$ for  every $\nu \in \mathcal{P}_{1}(\mathbb{R}^{d})$ and every $x \in \mathbb{R}^{d}$
\begin{equation*}
 | w_{\nu} (x) |  \leq M_{w} \big( 1 + | x| + m_{1}(\nu)\big)\,.
 \end{equation*}
 \end{itemize}
\end{lemma}

\begin{proof}
Item~$(i)$ follows from~$\mathrm{(v1)}$ and~$\mathrm{(v2)}$ and Corollary~\ref{Corollary}. Using that $\Delta(x, \nu) \in C_{\varepsilon}$, we have that $m_{1} ( ( id, \Delta)_{\#} \nu) \leq (1 + R_{\varepsilon}) m_{1}(\nu)$. Thus, we deduce~$(ii)$. 
\end{proof}

\begin{proposition}
\label{p:fast_good_N}
Let $v_{\Psi}$ satisfy~$\mathrm{(v1)}$--$\mathrm{(v3)}$ and let $F \colon \mathcal{P}_1(\mathbb{R}^d)\times \mathbb{R}^d \times ( 0 , +\infty) \times U \to (-\infty,+\infty]$ satisfy~$(\mathrm{F1})$--$(\mathrm{F5})$. Then, for every $\bar{\boldsymbol{x}} = (\bar{x}^{1}, \ldots, \bar{x}^{N}) \in ( \mathbb{R}^{d})^{N}$ there exists a unique solution $\boldsymbol{x}_{t} = (x^{1}_{t}, \ldots, x^{N}_{t})$ of the Cauchy problem~\eqref{limit_sym}. Moreover, if~$\delta>0$ and $\bar{\boldsymbol{x}} \in (B_{\delta})^{N}$, there exists $\varrho>0$, only depending on~$\delta$, such that $\boldsymbol{x}_{t} \in (B_{\varrho})^{N}$ for every $t \in [0, T]$. 
\end{proposition}

\begin{proof}
It is enough to notice that, by Lemma~\ref{l:W}, the velocity field $w_{\mu^{N}} (x^{i})$ with $\mu^{N} = \frac{1}{N} \sum_{i=1}^{N} \delta_{x_{i}}$ is locally Lipschitz and sublinear in $(\mathbb{R}^{d})^{N}$ for every $i = 1, \ldots, N$. Hence, system~\eqref{limit_sym} admits unique solution by standard ODE theory (see, e.g.,~\cite{Hale}). The boundedness of solutions can be obtained by Gr\"onwall inequality as in Theorem~\ref{Entropic_Problem_theorem}.
\end{proof}

The following convergence result holds for the $N$-particles system.

\begin{theorem}
\label{Newton}
Let $v_{\Psi}$ satisfy~$\mathrm{(v1)}$--$\mathrm{(v3)}$, let $F \colon \mathcal{P}_1(\mathbb{R}^d)\times \mathbb{R}^d \times ( 0 , +\infty) \times U \to (-\infty,+\infty]$ satisfy~$(\mathrm{F1})$--$(\mathrm{F5})$, let the operator~$\mathcal{T}_{\Psi}$ be defined as in~\eqref{fast_reaction}, and let~$G$ be as in~\eqref{Gfun}. For $N \in \mathbb{N}$, $\lambda \in (0, +\infty)$, and $\delta \in (0, +\infty)$, let $\boldsymbol{\bar{y}} = (\bar{y}^{1}, \ldots, \bar{y}^{N}) \in (B^{Y_\varepsilon}_{\delta})^N$ and, for $i = 1, \ldots, N$, let $t \mapsto \boldsymbol{y}_{\lambda, t} = (y^{1}_{\lambda, t}, \ldots, y^{N}_{\lambda, t})$ be the solution of the Cauchy problem~\eqref{two_time_scale_system} with initial datum $\bar{\boldsymbol{y}}$ and associated empirical measure~$\Lambda^{N}_{\lambda, t} = \frac{1}{N} \sum_{i=1}^{N} \delta_{y^{i}_{\lambda, t}}$. Moreover, let $t \mapsto  {\boldsymbol{x}}_{t} = ({x}^{1}_{t}, \ldots, {x}^{N}_{t})$ be the solution to~\eqref{limit_sym} with initial conditions $\bar{\boldsymbol{x}} = (\bar{x}^{1}, \ldots, \bar{x}^{N}) \in ( B_{\delta})^{N}$ and let
\begin{displaymath}
\boldsymbol{y}_{t} \coloneqq \big( (x^{1}_{t}, \Delta (x^{1}_{t}, \mu^{N}_{t})  ), \ldots,  (x^{N}_{t}, \Delta (x^{N}_{t}, \mu^{N}_{t}) ) \big)\,.
\end{displaymath}
Then, there exists $\chi_{\varepsilon, \delta}>0$ such that for every $t\in(0,T]$  
\begin{align}
\label{e:Nlambda_limit}
  \| \boldsymbol{y}_{\lambda,t} - \boldsymbol{{y}}_t \|_{\overline{Y}^N} & \leq \chi_{\varepsilon, \delta}  \Big( \frac{1}{\sqrt{\lambda}} +  e^{-\lambda \gamma_{\varepsilon} T}\Big) \qquad \text{if $p \in [1, 2]$} \,,
  \\
 \label{e:Nlambda_limit_2}  \| \boldsymbol{y}_{\lambda,t} - \boldsymbol{{y}}_t \|_{\overline{Y}^N} & \leq \chi_{\varepsilon, \delta} \Big( \Big(\frac{1}{\lambda}\Big)^{\frac{1}{p}} +  e^{-\frac{2 \lambda \gamma_{\varepsilon} T}{p}} \Big) \qquad \text{if $p \in (2, +\infty)$}
\end{align}
where $\gamma_{\varepsilon}>0$ is the constant introduced in Lemma~\ref{lemma}.
\end{theorem}

\begin{proof}
In what follows, we use also the notation $\boldsymbol{\ell}$ for a vector in~$(L^{p} (U, \eta))^{N}$ and we endow $(L^{p} (U, \eta))^{N}$ with the norm
$$
\| \boldsymbol{\ell}\|_{(L^{p} (U, \eta))^{N}} \coloneqq \frac{1}{N} \sum_{i=1}^{N} \| \ell^{i}\|_{L^{p} (U, \eta)}\,.
$$
Moreover, we set $\bell_{\lambda, t} \coloneqq ( \ell_{\lambda, t}^{1}, \ldots, \ell_{\lambda, t}^{N})$, $\bell_{t} \coloneqq ( \Delta (x^{1}_{t}, \mu^{N}_{t}) , \ldots, \Delta (x^{N}_{t}, \mu^{N}_{t}) )$, $\mu_{\lambda, t}^{N} \coloneqq \pi_{\#} \Lambda^{N}_{\lambda, t}$, and $\bar{\mu}^{N} \coloneqq \frac{1}{N} \sum_{i=1}^{N} \delta_{\bar{x}^{i}}$.

We provide a complete proof for $p \in [1, 2]$ and we highlight later on the main differences in the case $p \in (2, +\infty)$.  By~$(i)$ of Lemma~\ref{lemma} and by Proposition~\ref{p:fast_good_N}, there exists $\varrho>0$ such that $y^{i}_{\lambda, t}, y^{i}_{t} \in B_{\varrho}^{Y_{\varepsilon}}$ for $i =1, \ldots, N$ and $t \in [0, T]$. Hence, by triangle inequality and by \eqref{minimality_3} of Corollary~\ref{Corollary}, we have that
\begin{equation}
\label{e:hell}
\begin{split}
&
\quad\quad  \| \bell_{\lambda, t} - \bell_t \|_{(L^{p} (U, \eta))^{N}} 
\\
& 
\leq  \frac{1}{N}\sum_{i=1}^N \| {\ell}_{\lambda, t}^{i} - \Delta ( x^{i}_{\lambda,, t}, \mu^{N}_{\lambda, t}) \|_{L^{p}(U, \eta)} + \frac{1}{N}\sum_{i=1}^N \| \Delta( x^{i}_{\lambda, t}, \mu^{N}_{\lambda, t}) -\Delta( x^{i}_{t}, \mu^{N}_{t})  \|_{L^{p}(U, \eta)} 
    \\ 
    &
    \leq \frac{1}{N}\sum_{i=1}^N \| {\ell}_{\lambda, t}^{i} - \Delta ( x^{i}_{\lambda, t}, \mu^{N}_{\lambda, t}) \|_{L^{p}(U, \eta)}  + A_{\varepsilon, \varrho} ( \| \boldsymbol{x}_{\lambda,t} - \boldsymbol{x}_t \|_{(\mathbb{R}^d)^N} + \mathcal{W}_1( \mu_t^{N} , \hat{\mu}_t^{N}) ) 
    \\ 
    &
    \leq \frac{1}{N}\sum_{i=1}^N \| {\ell}_{\lambda, t}^{i} - \Delta ( x^{i}_{\lambda, t}, \mu^{N}_{\lambda, t}) \|_{L^{p}(U, \eta)}  + 2  A_{\varepsilon, \varrho}  \| \boldsymbol{x}_{\lambda, t} -\boldsymbol{x}_t \|_{(\mathbb{R}^d)^N} \,.
\end{split}
\end{equation}
Thanks to $(ii)$ of Lemma~\ref{lemma}, we may continue in~\eqref{e:hell} with
\begin{align}
\label{e:hell2}
 \| \bell_{\lambda, t} - \bell_{t} \|_{(L^{p}(U, \eta))^{N}} \leq    \omega_{\varepsilon, \delta} \bigg( \frac{1}{\sqrt{\lambda}} +  e^{-\lambda \gamma_{\varepsilon} T} \bigg) + 2  A_{\varepsilon, \varrho}  \| \boldsymbol{x}_{\lambda, t} -\boldsymbol{x}_t \|_{(\mathbb{R}^d)^N} \,. 
\end{align}
Combining~$\mathrm{(v1)}$, $\mathrm{(v2)}$, and inequality~\eqref{e:hell2}, we further estimate 
\begin{equation*}
    \begin{split}
        \frac{\de}{\de t} \| \boldsymbol{x}_{\lambda, t} - \boldsymbol{x}_{t} \|_{(\mathbb{R}^d)^N} & \leq \frac{1}{N} \sum_{i=1}^N | \dot{x}_{\lambda, t}^{i} - \dot{x}_t^{i} | 
        = \frac{1}{N} \sum_{i=1}^N | v_{{\Lambda}_{\lambda, t}^{N}} ( {x}_{\lambda, t}^{i},  \ell_{\lambda, t}^{i}) - w_{\mu^{N}_{t}} ( {x}_t^{i} ) | 
        \\ 
        &
        = \frac{1}{N} \sum_{i=1}^N \big | v_{{\Lambda}_{\lambda, t}^{N}} ( {x}_{\lambda, t}^{i},  \ell_{\lambda, t}^{i}) - v_{(id, \Delta)_{\#}\mu^{N}_{t}} ( {x}_t^{i} , \Delta(x^{i}_{t}, \mu^{N}_{t}) ) \big |
        \\
        &
        \leq \frac{2L_{\varrho}}{N} \sum_{i=1}^N   | x_{\lambda, t}^{i} - x_t^{i} | + \| \ell_{\lambda, t}^{i} - \Delta (x^{i}_{t}, \mu^{N}_{t})\|_{L^{p}(U, \eta)} 
        \\ 
        &
        \leq 2 L_{\varrho} (1 + 2 A_{\varepsilon, \varrho} ) \| \boldsymbol{x}_{\lambda, t} - \boldsymbol{x}_{t} \|_{(\mathbb{R}^d)^N} + 2L_{\varrho} \omega_{\varepsilon, \delta} \bigg( \frac{1}{\sqrt{\lambda}} +  e^{-\lambda \gamma_{\varepsilon} T} \bigg)\,.
    \end{split}
\end{equation*}
Equivantely, we can write 
\begin{equation*}
\begin{split}
    \frac{\de}{\de t} \bigg(  \| \boldsymbol{x}_{\lambda, t} - \boldsymbol{x}_t \|_{(\mathbb{R}^d)^N} & \ + \frac{ \omega_{\varepsilon, \delta}}{1 + 2 A_{\varepsilon, \varrho}} \bigg( \frac{1}{\sqrt{\lambda}} +  e^{-\lambda \gamma_{\varepsilon} T}  \bigg) \bigg)  
    \\
    &
    \leq 2 L_{\varrho} (1 + 2 A_{\varepsilon, \varrho} ) \bigg(  \| \boldsymbol{x}_{\lambda, t} - \boldsymbol{x}_{t} \|_{(\mathbb{R}^d)^N} + \frac{ \omega_{\varepsilon, \delta}}{1 + 2 A_{\varepsilon, \varrho}} \bigg( \frac{1}{\sqrt{\lambda}} +  e^{-\lambda \gamma_{\varepsilon} T} \bigg) \bigg) \,.
    \end{split}
    \end{equation*}
Therefore, by applying Gr\"{o}nwall's Lemma, for $\tau>0$ and $t \in [\tau, T]$ we obtain
\begin{equation}
\label{e:hell3}
\begin{split}
    \| \boldsymbol{x}_{\lambda, t} - \boldsymbol{x}_t \|_{(\mathbb{R}^d)^N} & 
    \leq \bigg( \|  \boldsymbol{x}_{\lambda, \tau} - \boldsymbol{x}_{\tau} \|_{(\mathbb{R}^d)^N} + \frac{ \omega_{\varepsilon, \delta}}{1 + 2 A_{\varepsilon, \varrho}} \bigg( \frac{1}{\sqrt{\lambda}} +  e^{-\lambda \gamma_{\varepsilon} T}  \bigg)\bigg) e^{2L_{\varrho} ( 1 + 2 A_{\varepsilon, \varrho}) (t - \tau ) } \,.
    \end{split}
\end{equation}
Recalling that $ \boldsymbol{x}_{\lambda, t}, \boldsymbol{x}_{ t} \in (B_{\varrho})^{N}$ for every $t \in [0, T]$, from $(\mathrm{v3})$ and $(ii)$ of Lemma~\ref{l:W}, we infer that
\begin{equation}
\label{e:hell5}
\|  \boldsymbol{x}_{\lambda, \tau} - \boldsymbol{x}_{\tau} \|_{(\mathbb{R}^d)^N} \leq ( M_{v} + M_{w}) (1 + 2\varrho + 2 R_{\varepsilon}) \tau\,.
\end{equation}
Thus, we deduce from~\eqref{e:hell3} that
\begin{equation*}
    \| \boldsymbol{x}_{\lambda, t} - \boldsymbol{x}_t \|_{(\mathbb{R}^d)^N} \leq \bigg ( ( M_{v} + M_{w}) (1 + 2\varrho + 2 R_{\varepsilon}) \tau + \frac{ \omega_{\varepsilon, \delta}}{1 + 2 A_{\varepsilon, \varrho}} \bigg( \frac{1}{\sqrt{\lambda}} +  e^{-\lambda \gamma_{\varepsilon} T}  \bigg)\bigg) e^{2L_{\varrho} ( 1 + 2 A_{\varepsilon, \varrho}) T }\,,
\end{equation*}
which, together with~\eqref{e:hell2}, yields~\eqref{e:Nlambda_limit} for $p \in [1, 2]$ by taking~$\tau = \frac{1}{\sqrt{\lambda}}$.

If $p \in (2, +\infty)$, we replace~\eqref{e:hell2} with
\begin{equation}
\label{e:hell4}
 \| \bell_{\lambda, t} - \bell_{t} \|_{(L^{p}(U, \eta))^{N}} \leq  2 (R_{\varepsilon} - r_{\varepsilon})^{\frac{p-2}{p}} \omega_{\varepsilon, \delta}^{\frac{2}{p}} \bigg( \frac{1}{\lambda} + e^{-2 \lambda \gamma_{\varepsilon} T} \bigg)^{\frac{1}{p}} + 2  A_{\varepsilon, \varrho}  \| \boldsymbol{x}_{\lambda, t} -\boldsymbol{x}_t \|_{(\mathbb{R}^d)^N} 
\end{equation}
Following step by step the argument for~\eqref{e:hell3} we get for $t \in [\tau, +\infty)$
\begin{equation*}
\begin{split}
    \| \boldsymbol{x}_{\lambda, t} - \boldsymbol{x}_t \|_{(\mathbb{R}^d)^N} & 
    \leq \bigg( \|  \boldsymbol{x}_{\lambda, \tau} - \boldsymbol{x}_{\tau} \|_{(\mathbb{R}^d)^N} + \frac{ 2 (R_{\varepsilon} - r_{\varepsilon})^{\frac{p-2}{p}} \omega_{\varepsilon, \delta}^{\frac{2}{p}}}{1 + 2A_{\varepsilon, \varrho}} \bigg( \frac{1}{\lambda} + e^{-2 \lambda \gamma_{\varepsilon} T} \bigg)^{\frac{1}{p}}  \bigg) e^{2L_{\varrho} ( 1 + 2 A_{\varepsilon, \varrho}) (t - \tau ) } \,.
    \end{split}
\end{equation*}
Then, \eqref{e:Nlambda_limit_2} follows from~\eqref{e:hell5} as in the case $p \in [1, 2]$ taking $\tau = \big( \frac{1}{\lambda}\big)^{\frac{1}{p}}$ and eventually re-defining the constant~$\chi_{\varepsilon, \delta}$.
\end{proof}

We introduce the {\em fast reaction} continuity equation
\begin{equation}
\label{e:fast-mean-field}
\partial_{t} \mu_{t} + \dive (w_{\mu_{t}} \mu_{t}) = 0, \qquad \mu_{0} = \bar{\mu},
\end{equation}
for $\bar{\mu} \in \mathcal{P}_{1}(\mathbb{R}^{d})$ and $\mu \in C([0, T];  ( \mathcal{P}_{1}(\mathbb{R}^{d}), \mathcal{W}_{1}))$. For the notion of Eulerian and Lagrangian solutions to~\eqref{e:fast-mean-field} we refer to Definitions~\ref{d:eulerian} and~\ref{d:lagrangian}, with the obvious modifications (see also~\cite{Gigli}). In the next proposition, we briefly discuss existence and uniqueness of solutions~\eqref{e:fast-mean-field}.

\begin{proposition}
\label{p:fast_good}
Let $v_{\Psi}$ satisfy~$\mathrm{(v1)}$--$\mathrm{(v3)}$ and let $F \colon \mathcal{P}_1(\mathbb{R}^d)\times \mathbb{R}^d \times ( 0 , +\infty) \times U \to (-\infty,+\infty]$ satisfy~$(\mathrm{F1})$--$(\mathrm{F5})$. Then, for every $\bar{\mu} \in \mathcal{P}_{c} (\mathbb{R}^{d})$ there exists a unique Eulerian (and Lagrangian) solution to~\eqref{e:fast-mean-field} with initial condition~$\bar{\mu}$. Moreover, for every $\delta>0$ and every $\bar{\mu}, \bar{\mu}^{n} \in \mathcal{P} (B_{\delta})$ such that $\mathcal{W}(\bar{\mu}^{n} , \bar{\mu}) \to 0$ as $n \to \infty$ we have that the corresponding solutions~$\mu, \mu^{n} \in C([0, T]; (\mathcal{P}_{1} (\mathbb{R}^{d}), \mathcal{W}_{1}))$ with initial conditions~$\bar{\mu}$ and~$\bar{\mu}^{n}$, respectively, satisfy
\begin{equation}
\lim_{n \to \infty} \, \mathcal{W}_{1} (\mu_{t}^{n}, \mu_{t} ) = 0 \qquad \text{uniformly for $t \in [0, T]$.}
\end{equation}
\end{proposition}

\begin{proof}
The thesis can be obtained by combining Lemma~\ref{l:W} with the arguments used in Theorems~\ref{Entropic_Main} and~\ref{thm:uniqueness}.
\end{proof}

\begin{remark}
\label{r:compact_support}
As a consequence of Proposition~\ref{p:fast_good}, we have that for every $\delta>0$ and every $\bar{\mu} \in \mathcal{P}(B_{\delta})$, there exists~$\varrho>0$ (only depending on~$\delta$ and $\varepsilon$) such that the solution~$\mu \in C([0, T]; (\mathcal{P}_{1}(\mathbb{R}^{d}) , \mathcal{W}_{1}))$ of~\eqref{e:fast-mean-field} with initial condition~$\bar{\mu}$ satisfies $\mathrm{spt} (\mu_{t}) \subseteq B_{\varrho}$ for every $t \in [0, T]$. This can be proven, for instance, by taking a sequence of empirical measures $\bar{\mu}^{N} \in \mathcal{P}(B_{\delta})$ such that $\mathcal{W}_{1}(\bar{\mu}^{N}, \bar{\mu}) \to 0$ and applying Propositions~\ref{p:fast_good_N} and~\ref{p:fast_good}.
\end{remark}

%

%


We are finally ready to discuss the convergence of the solutions to the continuity equations in the fast reaction limit.

\begin{theorem}
\label{t:final_mean_field}
Let $v_{\Psi}$ satisfy~$\mathrm{(v1)}$--$\mathrm{(v3)}$, let $F \colon \mathcal{P}_1(\mathbb{R}^d)\times \mathbb{R}^d \times ( 0 , +\infty) \times U \to (-\infty,+\infty]$ satisfy~$(\mathrm{F1})$--$(\mathrm{F5})$, let $\mathcal{T}_{\Psi}$ be defined as in~\eqref{fast_reaction}, let $\delta>0$, and let $\bar{\Lambda} \in \mathcal{P}(B_{\delta}^{Y_{\varepsilon}})$ and $\bar{\mu} = \pi_{\#} \bar{\Lambda} \in \mathcal{P}(B_{\delta})$. For every $\lambda>0$, let $\Lambda_{\lambda} \in C([0, T]; (\mathcal{P}_{1} (Y_{\varepsilon}) , \mathcal{W}_{1}))$ be the solution to~\eqref{e:continuity_lambda} with initial condition~$\bar{\Lambda}$ and let $\mu \in C([0, T]; (\mathcal{P}_{1} ( \mathbb{R}^{d} ) , \mathcal{W}_{1}))$ be the solution to~\eqref{e:fast-mean-field} with initial condition~$\bar{\mu}$. Then, for every $t \in [0, T]$ we have that
\begin{align}
\label{e:hell6}
& \mathcal{W}_{1} (\Lambda_{\lambda, t}, (id, \Delta)_{\#} \mu_{t})  \leq \chi_{\varepsilon, \delta}  \Big( \frac{1}{\sqrt{\lambda}} +  e^{-\lambda \gamma_{\varepsilon} T}\Big)  \qquad \text{if $p \in [1, 2]$,}\\
\label{e:hell6.1}
& \mathcal{W}_{1} (\Lambda_{\lambda, t}, (id, \Delta)_{\#} \mu_{t})  \leq \chi_{\varepsilon, \delta} \Big( \Big(\frac{1}{\lambda}\Big)^{\frac{1}{p}} +  e^{-\frac{2 \lambda \gamma_{\varepsilon} T}{p}} \Big) \qquad \text{if $p \in (2, +\infty)$,}
\end{align}
where $\gamma_{\varepsilon}$ and~$\chi_{\varepsilon, \delta}$ are the constants introduced in Lemma~\ref{lemma} and Theorem~\ref{Newton}, respectively.
\end{theorem}

\begin{proof}
We proceed by finite particles approximation and let us fix $\lambda \in (0, +\infty)$. Let us fix a sequence~$\bar{\by}_{N} \coloneqq (\bar{y}^{1}_{N}, \ldots, \bar{y}^{N}_{N}) \in (B_{\delta}^{Y_{\varepsilon}})^{N}$, let $\bar{\Lambda}^{N} \in \mathcal{P}(B^{Y_{\varepsilon}}_{\delta})$ denote the associated empirical measure, and assume that $\mathcal{W}_{1} (\bar{\Lambda}^{N} , \bar{\Lambda} ) \to 0$. Let us further denote by $\by_{\lambda,N, t} \in Y_{\varepsilon}^{N}$ the solution to~\eqref{two_time_scale_system} with initial condition $\bar{\by}_{N}$, let $\Lambda_{\lambda, t}^{N}$ be the corresponding empirical measure, let $\boldsymbol{x}_{N, t} \in (\mathbb{R}^{d})^{N}$ be the solution to~\eqref{limit_sym} with initial condition $\bar{\boldsymbol{x}}_{N}  = (x^{1}_{N}, \ldots, x^{N}_{N}) \in (B_{\delta})^{N}$, and finally let $\mu^{N}_{t}$ be the corresponding empirical measure. 

By triangle inequality, for every $N \in \mathbb{N}$ and every $t \in [0, T]$ we have that
\begin{equation}
\label{e:hell8}
\begin{split}
\mathcal{W}_{1} (\Lambda_{\lambda, t}, (id, \Delta)_{\#} \mu_{t}) \leq & \ \mathcal{W}_{1} (\Lambda_{\lambda, t}, \Lambda^{N}_{\lambda, t}) + \mathcal{W}_{1} (\Lambda^{N}_{\lambda, t}, (id, \Delta)_{\#} \mu^{N}_{t}) 
\\
&
+  \mathcal{W}_{1} ( (id, \Delta)_{\#}  \mu^{N}_{t}, (id, \Delta)_{\#} \mu_{t}) \,.
\end{split}
\end{equation}
By Corollary~\ref{c:lambda} we have that
\begin{equation}
\label{e:hell9}
\lim_{N \to \infty} \,  \mathcal{W}_{1} (\Lambda_{\lambda, t}, \Lambda^{N}_{\lambda, t}) = 0 \qquad \text{uniformly in $[0, T]$.}
\end{equation} 
Thanks to Remark~\ref{r:compact_support}, there exists~$\varrho>0$ such that $\mu^{N}_{t}, \mu_{t} \in \mathcal{P}(B_{\varrho})$ for every $t \in [0, T]$ and every $N \in \mathbb{N}$. Hence, by Proposition~\ref{p:fast_good} and by~\eqref{minimality_3} of Proposition~\ref{G} we have that
\begin{equation}
\label{e:hell10}
\lim_{N \to \infty} \,  \mathcal{W}_{1} ( (id, \Delta)_{\#}  \mu^{N}_{t}, (id, \Delta)_{\#} \mu_{t}) = 0 \qquad \text{uniformly in~$[0, T]$.}
\end{equation}
Applying Theorem~\ref{Newton} to $\by_{\lambda, N, t}$ and to 
\begin{displaymath}
\by_{N, t} \coloneqq \big( (x^{1}_{N, t}, \Delta (x^{1}_{N, t}, \mu^{N}_{t}) , \ldots,  (x^{N}_{N, t}, \Delta (x^{N}_{N, t} , \mu^{N}_{t}) \big)\,,
\end{displaymath}
we get that
\begin{align}
\label{e:hell11}
& \mathcal{W}_{1} (\Lambda^{N}_{\lambda, t}, (id, \Delta)_{\#} \mu^{N}_{t})  \leq \chi_{\varepsilon, \delta}  \Big( \frac{1}{\sqrt{\lambda}} +  e^{-\lambda \gamma_{\varepsilon} T}\Big) \qquad \text{if $p \in [1, 2]$,}\\
& \label{e:hell12} \mathcal{W}_{1} (\Lambda^{N}_{\lambda, t}, (id, \Delta)_{\#} \mu^{N}_{t}) \leq \chi_{\varepsilon, \delta} \Big( \Big(\frac{1}{\lambda}\Big)^{\frac{1}{p}} +  e^{-\frac{2 \lambda \gamma_{\varepsilon} T}{p}} \Big) \qquad \text{if $p \in (2, +\infty)$.}
\end{align}
Combining~\eqref{e:hell8}--\eqref{e:hell12} we infer~\eqref{e:hell6} and~\eqref{e:hell6.1}.
\end{proof}

\bigskip

\noindent\textbf{Acknowledgments}
The work of SA was partially funded by the Austrian Science Fund through the projects ESP-61 and P-35359.
The work of MM was partially supported by the \emph{Starting grant per giovani ricercatori} of Politecnico di Torino, by the MIUR grant Dipartimenti di Eccellenza 2018-2022 (E11G18000350001), and by the PRIN 2020 project \emph{Mathematics for industry 4.0 (Math4I4)} (2020F3NCPX) financed by the Italian Ministry of University and Research.
The work of FS was partially supported by the project \emph{Variational methods for stationary and evolution problems with singularities and interfaces} PRIN 2017 (2017BTM7SN) financed by the Italian Ministry of Education, University, and Research and by the project Starplus 2020 Unina Linea 1 \emph{New challenges in the variational modeling of continuum mechanics} from the University of Naples ``Federico II'' and Compagnia di San Paolo (CUP: E65F20001630003).
MM and FS are members of the GNAMPA group of INdAM.
This work stems for the Master's Degree thesis of CDE at Politecnico di Torino, defended on 15 July 2022.

\bibliographystyle{siam}
\bibliography{ADEMS_references.bib}

\end{document}